\documentclass[a4paper]{article}
\usepackage[margin=30mm]{geometry}
\usepackage{amsmath}
\usepackage{amssymb}
\usepackage{amsthm}
\usepackage{bbm}
\usepackage{braket}
\usepackage{titlesec}
\usepackage{titlefoot}
\usepackage{bm}
\usepackage{color}
\usepackage{comment}
\titleformat*{\section}{\normalsize\bfseries}
\titleformat*{\subsection}{\normalsize\bfseries}
\allowdisplaybreaks
\def\R{\mathbb{R}}
\def\e{{\varepsilon}}        

\def\p{\partial}

\newtheorem{thm}{Theorem}[section]
\newtheorem{lem}[thm]{Lemma}
\newtheorem{cor}[thm]{Corollary}
\newtheorem{prop}[thm]{Proposition}
\newtheorem{rem}[thm]{Remark}

\makeatletter
\@addtoreset{equation}{section}
 
  \makeatother
 
\begin{document}

\title{
\vspace{-1cm}
\large{\bf Large Time Behavior and Optimal Decay Estimate for Solutions to the Generalized Kadomtsev--Petviashvili--Burgers Equation in 2D}}
\author{Ikki Fukuda and Hiroyuki Hirayama 
}
\date{}
\maketitle


\vspace{-0.75cm}
\begin{abstract}
We consider the Cauchy problem for the generalized Kadomtsev--Petviashvili--Burgers equation in 2D. This is one of the nonlinear dispersive-dissipative type equations, which has a spatial anisotropic dissipative term. Under some suitable regularity assumptions on the initial data $u_{0}$, especially the condition $\p_{x}^{-1}u_{0} \in L^{1}(\R^{2})$, it is known that the solution to this problem decays at the rate of $t^{-\frac{7}{4}}$ in the $L^{\infty}$-sense. In this paper, we investigate the more detailed large time behavior of the solution and construct the approximate formula for the solution at $t\to \infty$. Moreover, we obtain a lower bound of the $L^{\infty}$-norm of the solution and prove that the decay rate $t^{-\frac{7}{4}}$ of the solution given in the previous work to be optimal. 
\end{abstract}


\section{Introduction}  

We consider the Cauchy problem for the generalized Kadomtsev--Petviashvili--Burgers equation:
\begin{align}\label{KPB}
\begin{split}
& u_{t} + u^{p}u_{x} + u_{xxx} + \e v_{y} -\nu u_{xx} = 0, \ \ v_{x}=u_{y}, \quad (x, y) \in \R^{2}, \ t>0,\\
& u(x, y, 0) = u_{0}(x, y), \quad (x, y) \in \R^{2}, 
\end{split}
\end{align}
where $u = u(x, y, t)$ is a real-valued unknown function, $u_{0}(x, y)$ is a given initial data, $p\ge2$ is an integer, $\e \in \{-1, 1\}$ and $\nu>0$. 
The subscripts $x$, $y$ and $t$ denote the partial derivatives with respect to $x$, $y$ and $t$, respectively. 
The purpose of our study is to analyze the large time behavior of the solutions to \eqref{KPB}. 
In particular, we focus on the asymptotic decay of the solutions. More precisely, we would like to discuss the optimality for the time decay estimate of the solutions. 

First of all, let us recall some known results related to this problem. If we take $\nu=0$ in \eqref{KPB}, we obtain the generalized Kadomtsev--Petviashvili equation: 
\begin{align}\label{gKP}
\begin{split}
& u_{t} + u^{p}u_{x} + u_{xxx} + \e v_{y} = 0, \ \ v_{x}=u_{y}, \quad (x, y) \in \R^{2}, \ t>0,\\
& u(x, y, 0) = u_{0}(x, y), \quad (x, y) \in \R^{2}.
\end{split}
\end{align}
This equation has the conserved mass and energy:
\[
M(u)=\int_{\R^2}u^2dxdy,\quad 
E(u)=\int_{\R^2}\left(\frac{1}{2}u_x^2-\frac{\e}{2}v^2-\frac{1}{(p+1)(p+2)}u^{p+2}\right)dxdy.
\]
Thus, natural energy space can be defined by
\[
E^1(\R^2):=\left\{u\in L^2(\R^2); \ \|u\|_{L^2}+\|u_x\|_{L^2}+\|v\|_{L^2}<\infty,\ v_x=u_y\right\}. 
\]
When $p=1$, \eqref{gKP} is classically called the KP-I ($\e=-1$) or KP-I\hspace{-.1em}I ($\e=1$) equation. 
Ionescu--Kenig--Tataru \cite{IKT08} proved  
the global well-posedness of the KP-I equation in $E^1(\R^2)$. 
On the other hand, Bourgain \cite{B93} proved that the 
KP-I\hspace{-.1em}I equation is globally well-posed in $L^2(\R^2)$. 
Furthermore, Hadac--Herr--Koch \cite{HHK09} 
showed the global well-posedness of the KP-I\hspace{-.1em}I equation 
for small initial data in anisotropic Sobolev space $\dot{H}^{-\frac{1}{2},0}(\R^{2})$, 
which is the scaling critical space. 
They also proved that the scattering of the solution for small initial data. 
The same methods by Bourgain \cite{B93} and Hadac--Herr--Koch \cite{HHK09} 
cannot be applied for the KP-I equation 
because the KP-I equation does not have good structure such as the KP-I\hspace{-.1em}I equation from the point of view of the non-resonance. 
The large time behavior of the solution to 
the KP-I and KP-I\hspace{-.1em}I equations are studied by Hayashi--Naumkin \cite{HN14}. 
They gave the decay estimates and the asymptotics of $u_x$ 
for small initial data, 
where $u$ is a unique global solution to the KP-I or KP-I\hspace{-.1em}I equations. 
The scattering result for the KP-I equation is obtained by Harrop-Griffiths--Ifrim--Tataru in \cite{HIT16}. 
For the generalized KP equation, 
Hayashi--Naumkin--Saut \cite{HNS99} 
considered the case $p\in \mathbb{Z}$ with $p\ge 2$ in \eqref{gKP}. 
They gave the decay estimates and the asymptotics of 
solution for small initial data. 
The decay estimates are improved by Niizato \cite{N11}. 
The results in \cite{HN14, HNS99, N11}, 
the decaying and the regularity of $\partial_x^{-1}u_0$ are imposed 
for initial data $u_0$, where $\partial_x^{-1}$ denotes the 
anti-derivative operator which will be defined by \eqref{anti-derivative} below. 
But the assumption for $\partial_x^{-1}u_0$ is not unnatural 
because of the form of the corresponding energy. 

Compared with \eqref{gKP}, Eq.~\eqref{KPB}, which has the dissipative term $-\nu u_{xx}$, 
should be considered as a dispersive-dissipative type equation, not as a purely dispersive type equation. 
However, especially in higher dimensions, there are very few results on the large time behavior of the solutions to dispersive-dissipative type equations with spatial anisotropic dissipative term like that of \eqref{KPB}. 
Most of the known results of dispersive-dissipative type equations are for the one-dimensional case. 
First, let us introduce some related results on the following Cauchy problem for the generalized KdV--Burgers equation, which is a one-dimensional version for \eqref{KPB}: 
\begin{align}\label{KdVB}
\begin{split}
& u_{t} + u^{p}u_{x} + u_{xxx} -\nu u_{xx} = 0, \quad x \in \R, \ t>0,\\
& u(x, 0) = u_{0}(x), \quad x \in \R.
\end{split}
\end{align}
For the above problem \eqref{KdVB}, since the global well-posedness in some Sobolev spaces can be easily proved by virtue of the dissipation effect, the main research topic is about the large time behavior of the solution. This topic was first studied by Amick--Bona--Schonbek \cite{ABS89}. They derived the time decay estimates of the solution, in the case of $p=1$. In particular, they showed that if $u_{0}\in H^{2}(\R) \cap L^{1}(\R)$, then the solution satisfies the following estimate: 
\begin{equation}\label{KdVB-decay}
\left\|u(\cdot, t)\right\|_{L^{q}} \le Ct^{-\frac{1}{2}\left(1-\frac{1}{q}\right)}, \ \ t>0, \ \ 2\le q\le \infty. 
\end{equation}
We note that the decay rate of the solution given by \eqref{KdVB-decay} is the same as the one for the solution to the linear heat equation: 
\[
u_{t}-\nu u_{xx}=0, \quad x \in \R, \ t>0. 
\] 
It means that the dissipation effect is stronger than the effects of the dispersion and the nonlinearity. 
Moreover, Karch \cite{K99-2} studied \eqref{KdVB} with $p=1$ in more details and extended the results given in \cite{ABS89}. Actually, if $u_{0}\in H^{1}(\R)\cap L^{1}(\R)$, then the following asymptotic formula can be established: 
\begin{equation}\label{Karch}
\lim_{t\to \infty}t^{\frac{1}{2}\left(1-\frac{1}{q}\right)}\left\|u(\cdot, t)-\chi(\cdot, t)\right\|_{L^{q}}
=0, \quad 1\le q\le \infty,
\end{equation}
where $\chi(x, t)$ is the self-similar solution to the following Burgers equation: 
\begin{equation*}
\chi_{t}+\chi \chi_{x}-\nu \chi_{xx}=0, \ \ x\in \R, \ t>0. 
\end{equation*}
This self-similar solution $\chi(x, t)$ can be obtained explicitly as follows (cf.~\cite{C51, H50}): 
\begin{equation*}
\chi(x, t):=\frac{1}{\sqrt{t}} \chi_{*} \left(\frac{x}{\sqrt{t}}\right), \ \ x\in \R,\ \ t>0,
\end{equation*}
where
\begin{equation*}
\chi_{*}(x):=\frac{\sqrt{\nu}\left(e^{\frac{M}{2\nu}}-1\right)e^{-\frac{x^{2}}{4\nu}}}{\sqrt{\pi}+\left(e^{\frac{M}{2\nu}}-1\right)\int_{x/\sqrt{4\nu}}^{\infty}e^{-y^{2}}dy}, 
\quad M:=\int_{\R}u_{0}(x)dx\neq0. 
\end{equation*}
Moreover, Hayashi--Naumkin \cite{HN06} improved the asymptotic rate given in \eqref{Karch} of the solution $u(x, t)$ to the above self-similar solution $\chi(x, t)$. More precisely, they showed that if we additionally assume $xu_{0}\in L^{1}(\R)$, then $u(x, t)$ tends to $\chi(x, t)$ at the rate of $t^{-\frac{1}{2}-\delta}$ in the $L^{\infty}$-sense, where $\delta\in (0, 1/2)$. Furthermore, Kaikina--Ruiz-Paredes \cite{KP05} succeeded to construct the second asymptotic profile of the solution, under the condition $xu_{0}\in L^{1}(\R)$. In view of the second asymptotic profile, they proved that the optimal asymptotic rate to $\chi(x, t)$ is $t^{-1}\log t$ in the $L^{\infty}$-sense, and also showed that the effect of the dispersion term $u_{xxx}$ appears from the second term of its asymptotics. For some related results about this problem, we can also refer to \cite{F19-1, F19-2, HKNS06}. When $p\ge2$ in \eqref{KdVB}, we can say that the nonlinearity is weak compared with the case of $p=1$, because if the solution $u(x, t)$ decays, then $u^{p}u_{x}$ decays fast. Therefore, the main part of the solution $u(x, t)$ is governed by the heat kernel:
\[
G(x, t):=\frac{1}{\sqrt{4\pi \nu t}}\exp\left(-\frac{x^{2}}{4\nu t}\right), \ \ x\in \R, \ t>0. 
\]
Actually, if $p\ge2$ and $u_{0}\in H^{1}(\R)\cap L^{1}(\R)$, then we are able to obtain the same decay estimate \eqref{KdVB-decay} and the following asymptotic formula: 
\begin{equation}\label{Karch-heat}
\lim_{t\to \infty}t^{\frac{1}{2}\left(1-\frac{1}{q}\right)}\left\|u(\cdot, t)-MG(\cdot, t)\right\|_{L^{q}}
=0, \quad 1\le q\le \infty. 
\end{equation}
The above formula \eqref{Karch-heat} can be found in e.g.~\cite{K99-1, K97-2} by Karch. Especially in \cite{K99-1}, he studied \eqref{KdVB} in details when $p\ge2$ (in fact, for more general nonlinearities are treated), and constructed the second asymptotic profiles of the solution. Indeed, it was shown in \cite{K99-1} that if we additionally assume $xu_{0}\in L^{1}(\R)$, then the optimal asymptotic rates to $MG(x, t)$ are given by $t^{-1+\frac{1}{2q}}$ when $p>2$ and $t^{-1+\frac{1}{2q}}\log t$ when $p=2$, respectively, in the $L^{q}$-sense. Furthermore, according to his results, it can be seen that the effects of the nonlinear term $u^{p}u_{x}$ and the dispersion term $u_{xxx}$ appear from the second asymptotic profile of the solution, in the case of $p>2$. However, we are able to find that the effect of 
dispersion term does not appear in the second asymptotic profile when $p=2$. The results introduced in this paragraph also hold for more general dispersive-dissipative type equations which has a similar structure to \eqref{KdVB}, such as the (generalized) BBM--Burgers equation: 
\[
u_{t}-u_{xxt}+u^{p}u_{x} + u_{xxx} -\nu u_{xx}=0, \ \ x\in \R, \ t>0. 
\]
For details, see e.g.~\cite{ABS89, HKN07, HKNS06, K97-1, K97-2, K99-1} and references there in. 

The asymptotic formulas \eqref{Karch} and \eqref{Karch-heat} tell us that the dispersion term $u_{xxx}$ can be negligible in the sense of the first approximation, due to the dissipation effect. From this point of view, we can expect the same phenomenon also to occur in our target equation \eqref{KPB}. However, \eqref{KPB} has a spatial anisotropy and the dissipative term $-\nu u_{xx}$ works only in the $x$-direction. 
Moreover, we also need to consider the effect of the dispersion term $\e v_{y}$ in the $y$-direction. 
Therefore, the method used in one-dimensional problems cannot be applied to the analysis for \eqref{KPB}. Of course, the method used for dispersive type equations like \eqref{gKP} also cannot be applied to \eqref{KPB}.
For this reason, it is difficult to analyze the structure of the solution to \eqref{KPB}, 
and there is only a previous result \cite{M99} on the large time behavior of the solution, up to the authors knowledge (we will introduce it later). Therefore, if we can obtain some new perspectives on \eqref{KPB}, then it is expected that those results help us to construct a general theory for spatial anisotropic dispersive-dissipative type equations.

Next, we would like to explain about the known results for \eqref{KPB} given by Molinet \cite{M99}. Before doing that, in order to analyze \eqref{KPB}, for $m \in \mathbb{N}$ and $s\ge0$, let us introduce some function spaces: 
\begin{align}
& \dot{\mathbb{H}}^{-m}_{x}(\R^{2}) := \left\{ f \in \mathcal{S}'(\R^{2}); \ \|f\|_{\dot{\mathbb{H}}^{-m}_{x}} := \left\| \xi^{-m} \hat{f} \right\|_{L^{2}} < \infty \right\},  \label{space-Hm} \\
& X^{s}(\R^{2}) := \left\{ f \in H^{s}(\R^{2}); \ \|f\|_{X^{s}} := \|f\|_{H^{s}} + \left\| \mathcal{F}^{-1}\left[ \xi^{-1}\hat{f}(\xi, \eta)\right] \right\|_{H^{s}} <\infty \right\}.  \label{space-Xs}
\end{align}
Moreover, we shall define the anti-derivative operator $\p_{x}^{-1}$ acting on $\dot{\mathbb{H}}^{-m}_{x}(\R^{2})$ by 
\begin{equation}\label{anti-derivative}
\p_{x}^{-1}f(x, y):=\mathcal{F}^{-1}\left[ (i\xi)^{-1}\hat{f}(\xi, \eta) \right](x, y).
\end{equation}
We can easily to see that $X^{s}(\R^{2})$ is embedded in $\dot{\mathbb{H}}^{-m}_{x}(\R^{2})$ 
for any $m\in \mathbb{N}$ and $s\ge 0$. 
Therefore, for $u \in X^{1}(\R^{2})$, $v_{x}=u_{y}$ can be rewritten by $v = \p_{x}^{-1}u_{y}$. 
Thus, \eqref{KPB} is equivalent to the following Cauchy problem: 
\begin{align}\label{KPB-s}
\begin{split}
& u_{t} + u^{p}u_{x} + u_{xxx} + \e \p_{x}^{-1}u_{yy} -\nu u_{xx} = 0, \quad (x, y) \in \R^{2}, \ t>0,\\
& u(x, y, 0) = u_{0}(x, y), \quad (x, y) \in \R^{2}. 
\end{split}
\end{align}
In what follows, we would like to consider \eqref{KPB-s} instead of \eqref{KPB}. 
Now, let us recall some known results about the Cauchy problem \eqref{KPB-s}. 
First of all, we shall introduce the fundamental results about the global well-posedness of \eqref{KPB-s}. Molinet~\cite{M99} constructed the solutions to \eqref{KPB}, provided the initial data $u_{0} \in X^{s}(\R^{2})$ for $s>2$. More precisely, he showed that if $p\ge1$ and $\left\|u_{0}\right\|_{H^{1}}+\left\|\p_{x}^{2}u_{0}\right\|_{L^{2}}+\left\|\p^{-1}_{x}\p_{y}u_{0}\right\|_{L^{2}}$ is sufficiently small, then \eqref{KPB-s} has a unique global solution $u(x, y, t)$ satisfying the following properties: 
\begin{align*}
&u \in C_{b}([0, \infty); X^{s}(\R^{2})) \cap L^{2}(0, \infty; H^{s}(\R^{2})), \\
&\p_{x}^{l}u \in C_{b}([1, \infty); H^{s}(\R^{2})) \cap L^{2}(1, \infty; H^{s}(\R^{2})), \ \ l \in \mathbb{N}. 
\end{align*}
For more details, see Theorem~\ref{thm.u-SDGE} below. Moreover, for the related results about the global well-posedness of \eqref{KPB-s}, we can also refer to another his paper \cite{M00}. Furthermore, for the case $p=1$, 
we also remark that the global well-posedness of (\ref{KPB-s}) is improved
by Kojok~\cite{K07} for $\e =1$ and Darwich~\cite{D12} for $\e =-1$ (see, also \cite{MR02}). 

In addition to the global well-posedness, Molinet also discussed the large time behavior of the solution to \eqref{KPB-s} in \cite{M99}. More precisely, he gave some results about the time decay estimates for the solution (see Propositions~\ref{prop.u-decay-L2}, \ref{prop.u-decay-L2-high} and \ref{prop.u-decay-Linf} below). Especially, under the additional assumptions $p\ge2$, $u_{0} \in X^{s}(\R^{2})$ for $s\ge3$, $\p_{x}^{-1}u_{0} \in L^{1}(\R^{2})$ and $\p_{x}^{-1}\p_{y}^{m}u_{0} \in L^{1}(\R^{2})$ for some integer $m\ge3$, if $\p_{x}^{-1}\p_{y}^{j}u_{0} \in L^{1}(\R^{2})$ with $j \in \left\{0, 1, \cdots, \min\{[s], m\}-3\right\}$, then the solution to \eqref{KPB-s} satisfies 
\begin{equation}\label{u-decay-Molinet}
\left\|\p_{x}^{l}\p_{y}^{j}u(\cdot, \cdot, t)\right\|_{L^{\infty}} \le Ct^{-\frac{7}{4}-\frac{l}{2}}, \ \ t\ge1, 
\end{equation}
for all non-negative integer $l$. Here, we note that this decay estimate \eqref{u-decay-Molinet} is a different from not only \eqref{KdVB-decay} but also the one for the solution to two-dimensional parabolic type equations. Actually, this estimate \eqref{u-decay-Molinet} is constructed by the interactions of the dissipation $-\nu u_{xx}$ in the $x$-direction, dispersion $\e \p_{x}^{-1}u_{yy}$ in the $y$-direction and the anti-derivative $\p_{x}^{-1}$ on the initial data $u_{0}$. We also remark that the different decay estimate can be obtained in the case of $p=1$ (for details, see Theorem~5.2 in \cite{M99}). However, we do not treat it in this paper, because this case is some what complicated to discuss the more detailed large time behavior. 

In \cite{M99}, Molinet also mentioned about the optimality for the decay rate $t^{-\frac{7}{4}}$ given in \eqref{u-decay-Molinet}. 
To state such a result, we introduce the following linearized Cauchy problem for \eqref{KPB-s}: 
\begin{align}\label{LKPB}
\begin{split}
& \tilde{u}_{t} + \tilde{u}_{xxx} + \e \p_{x}^{-1}\tilde{u}_{yy} -\nu \tilde{u}_{xx} = 0, \ \ (x, y) \in \R^{2}, \ t>0, \\
& \tilde{u}(x, y, 0) = u_{0}(x, y), \ \ (x, y) \in \R^{2}. 
\end{split}
\end{align}
For this problem, he proved that if we assume $u_{0}\in X^{1}(\R^{2})$ and $(1+y^{4})\p_{x}^{-1}u_{0}\in L^{1}(\R^{2})$, then the following results hold for almost every $y \in \R$ and all $t\ge1$: 
\begin{equation}\label{Molinet-lemmas}
\lim_{t\to \infty}t^{3}\int_{\R}\tilde{u}^{2}(x, y, t)dx=\frac{\sqrt{\pi}}{16\nu^{2}}\left(\int_{\R^{2}}\p_{x}^{-1}u_{0}(x, y)dxdy\right)^{2}, \quad 
t^{\frac{5}{4}}\int_{\R}\left|\tilde{u}(x, y, t)\right|dx \le C(y),  
\end{equation}
where $C(y)$ is a certain positive constant depending on $y$. In particular, there exists $y_{0} \in \R$ such that the both results of Eq.~\eqref{Molinet-lemmas} hold for $y=y_{0}$. Therefore, we obtain
\begin{equation}\label{Molinet-pre}
\lim_{t\to \infty}t^{3}\int_{\R}\tilde{u}^{2}(x, y_{0}, t)dx
\le \left(\liminf_{t\to \infty}t^{\frac{7}{4}}\left\|\tilde{u}(\cdot, \cdot, t)\right\|_{L^{\infty}}\right)
\left(\limsup_{t\to \infty}t^{\frac{5}{4}}\int_{\R}\left|\tilde{u}(x, y_{0}
, t)\right|dx\right). 
\end{equation}
Therefore, combining \eqref{Molinet-lemmas} and \eqref{Molinet-pre}, we get 
\begin{equation}\label{Molinet-linear}
\liminf_{t\to \infty}t^{\frac{7}{4}}\left\|\tilde{u}(\cdot, \cdot, t)\right\|_{L^{\infty}}
\ge \frac{\sqrt{\pi}}{16\nu^{2}}\left(\int_{\R^{2}}\p_{x}^{-1}u_{0}(x, y)dxdy\right)^{2}C(y_{0})^{-1}. 
\end{equation}
In \cite{M99}, the optimality for the decay rate $t^{-\frac{7}{4}}$ given in \eqref{u-decay-Molinet} is concluded by based on the above estimate \eqref{Molinet-linear}. 
However, \eqref{Molinet-linear} is the estimate for the solution $\tilde{u}(x, y, t)$ to the linearized problem \eqref{LKPB} not for the original solution $u(x, y, t)$ to the nonlinear problem \eqref{KPB-s}. 
In order to conclude that the decay rate $t^{-\frac{7}{4}}$ given in \eqref{u-decay-Molinet} is optimal, it is necessary to evaluate $t^{\frac{7}{4}}\left\|u(\cdot, \cdot, t)\right\|_{L^{\infty}}$ uniformly from below. In addition, it seems that the assumption $(1+y^{4})\p_{x}^{-1}u_{0}\in L^{1}(\R^{2})$ is a little bit strong because the weight $1+y^{4}$ is not natural.  
Therefore, we would like to re-discuss them and derive the lower bound of the $L^{\infty}$-norm of the original solution $u(x, y, t)$ to \eqref{KPB-s}, without any weight assumption on the initial data $u_{0}$.

\medskip

\par\noindent
\textbf{\bf{Main Results.}} 
Now, we state our main results. First of all, let us introduce the following problem: 
\begin{align}\label{w-asymp}
\begin{split}
& w_{t} + \e \p_{x}^{-1}w_{yy} -\nu w_{xx} = -u^{p}u_{x}, \quad (x, y) \in \R^{2}, \ t>0,\\
& w(x, y, 0) = u_{0}(x, y), \quad (x, y) \in \R^{2}, 
\end{split}
\end{align}
where $u(x, y, t)$ is the original solution to \eqref{KPB-s}. Under some appropriate regularity assumptions on the initial data $u_{0}$, the solution $w(x, y, t)$ to \eqref{w-asymp} can be expressed by 
\begin{equation}\label{w-sol}
w(t)=K(t)*u_{0} -\frac{1}{p+1}\int_{0}^{t} \p_{x}K(t-\tau)*\left(u^{p+1}(\tau)\right) d\tau, 
\end{equation}
where $K(x, y, t)$ is defined by 
\begin{align}
& K(x, y, t) := t^{ -\frac{5}{4} } K_{*} \left( xt^{-\frac{1}{2}}, yt^{-\frac{3}{4}} \right), \quad (x, y)\in \R^{2}, \ \ t>0,  \label{DEF-K}\\
& K_{*}(x, y) := \frac{ 1 }{ 4\pi^{ \frac{3}{2} } \nu^{\frac{3}{4}} } \int_{0}^{\infty} r^{-\frac{1}{4} }e^{ -r } \cos \left( x \sqrt{\frac{r}{\nu}} +\frac{y^{2}}{4\e}\sqrt{\frac{r}{\nu}} - \frac{\pi}{4}\e \right) dr, \quad (x, y)\in \R^{2}.  \label{DEF-K*}
\end{align}
Here, we note that $K(x, y, t)$ is the integral kernel of the solution to the linearized problem for \eqref{w-asymp} (for details, see \eqref{w-linear} and \eqref{ap-derivation} below). By using the decay estimate for $K(x, y, t)$ (Proposition~\ref{prop.L-decay-K} below), under the same assumptions on the known result given in \cite{M99} (for details, see the assumptions of Proposition~\ref{prop.u-decay-Linf} below), we can see that the original solution $u(x, y, t)$ to \eqref{KPB-s} is well approximated by the above function $w(x, y, t)$. More precisely, we have the following result: 
\begin{thm}\label{thm.u-asymp}
Let $p\ge2$ be an integer and $s\ge3$. Suppose that $u_{0} \in X^{s}(\R^{2})$, $\p_{x}^{-1}u_{0} \in L^{1}(\R^{2})$, $\p_{x}^{-1}\p_{y}^{m}u_{0} \in L^{1}(\R^{2})$ for some integer $m\ge3$, and $\left\|u_{0}\right\|_{H^{1}}+\left\|\p_{x}^{2}u_{0}\right\|_{L^{2}}+\left\|\p^{-1}_{x}\p_{y}u_{0}\right\|_{L^{2}}$ is sufficiently small. If $\p_{x}^{-1}\p_{y}^{j}u_{0} \in L^{1}(\R^{2})$ with $j \in \left\{0, 1, \cdots, \min\{[s], m\}-3\right\}$, then the solution $u(x,y,t)$ to \eqref{KPB-s} satisfies 
\begin{equation}
\left\|\p_{x}^{l}\p_{y}^{j}\left(u(\cdot, \cdot, t)-w(\cdot, \cdot, t)\right)\right\|_{L^{\infty}}\le Ct^{-\frac{9}{4}-\frac{l}{2}}, \ \ t\ge2, \label{u-asymp}
\end{equation}
for all non-negative integer $l$, where $w(x, y, t)$ is given by \eqref{w-sol}. 
\end{thm}

\begin{rem}
{\rm 
We note that $w_{xxx}$ is not included in \eqref{w-asymp}, i.e. the effect of the dispersion term does not appear in the integral kernel $K(x, y, t)$. 
In this sense, the above approximate formula \eqref{u-asymp} tells us that the dispersion effect in the direction in which the dissipative term is working can be negligible as $t\to \infty$. 
}
\end{rem}

Moreover, we have succeeded to get a lower bound of the $L^{\infty}$-norm of the main part of $w(x, y, t)$ (see Proposition~\ref{prop.w-main-lower} below). Therefore, combining that estimate and the above Theorem~\ref{thm.u-asymp}, we are able to obtain the following lower bound of the $L^{\infty}$-norm of the solution $u(x, y, t)$ to \eqref{KPB-s}: 
\begin{thm}\label{thm.u-est-lower}
Under the same assumptions in Theorem~\ref{thm.u-asymp}, there exist positive constants $C_{\dag}>0$ and $T_{\dag}>0$ such that the solution $u(x, y, t)$ to \eqref{KPB-s} satisfies 
\begin{equation}\label{u-est-lower}
\left\|\p_{x}^{l}\p_{y}^{j}u(\cdot, \cdot, t)\right\|_{L^{\infty}} \ge C_{\dag}\left|\mathcal{M}_{j}\right|t^{-\frac{7}{4}-\frac{l}{2}}, \ \ t \ge T_{\dag}, 
\end{equation}
for all non-negative integer $l$, where the constant $\mathcal{M}_{j}$ is defined by 
\begin{equation}\label{DEF-mathM}
\mathcal{M}_{j} := \int_{\R^{2}}\p_{x}^{-1}\p_{y}^{j}u_{0}(x, y)dxdy-\frac{1}{p+1}\int_{0}^{\infty}\int_{\R^{2}}\p_{y}^{j}\left(u^{p+1}(x, y, t)\right)dxdydt. 
\end{equation}
\end{thm}

\begin{rem}
{\rm 
By virtue of Theorem~\ref{thm.u-est-lower}, under the assumptions in Theorem~\ref{thm.u-asymp} and $\mathcal{M}_{j}\neq0$, we can conclude that the $L^{\infty}$-decay estimate \eqref{u-decay-Molinet} is optimal with respect to the time decaying order. 
We will give a sufficient condition for $\mathcal{M}_{j}\neq0$ in Section~3 (see Remarks \ref{rem-M} and \ref{rem-M2} below). 
Furthermore, we note that the weight condition for the initial data $u_{0}$ such as in \cite{M99} is not necessary in our result. 
}
\end{rem}

The rest of this paper is organized as follows. In Section~2, we analyze the linearized problem for \eqref{KPB-s} and prove some decay estimates for the solution. Moreover, we derive the leading term of its solution and construct the corresponding asymptotic formula. 
The proofs of our main results Theorems~\ref{thm.u-asymp} and \ref{thm.u-est-lower} are given in Section~3. This section is divided into three subsections. In Section~3.1, we prepare a couple of basic lemmas including the known results given by Molinet \cite{M99}. After that, we prove Theorem~\ref{thm.u-asymp} in Section~3.2. Finally, we give the proof of Theorem~\ref{thm.u-est-lower} in Section~3.3. 
We can say that this Section~3.3 is the main part of the proofs, because Theorem~\ref{thm.u-est-lower} is the most difficult to prove. In order to show it, we need to consider the both effects of the linear part and the nonlinear part, to get a lower bound of the solution $u(x, y, t)$. 

\medskip

\par\noindent
\textbf{\bf{Notations.}} 
We define the Fourier transform of $f$ and the inverse Fourier transform of $g$ as follows:
\begin{align*}
\begin{split}
&\hat{f}(\xi, \eta) = \mathcal{F}[f](\xi, \eta) := \frac{1}{2\pi}\int_{\R^{2}}e^{-ix\xi-iy\eta}f(x, y)dxdy, \\
&\check{g}(x, y) = \mathcal{F}^{-1}[g](x, y):=\frac{1}{2\pi}\int_{\R^{2}}e^{ix\xi+iy\eta}g(\xi, \eta)d\xi d\eta.
\end{split}
\end{align*}

For $1\le p\le \infty$ and $s\ge 0$, $L^{p}(\R^{2})$ and $H^{s}(\R^{2})$ denote the Lebesgue spaces and the Sobolev spaces, respectively. Moreover, we also use the additional function spaces $\dot{\mathbb{H}}^{-m}_{x}(\R^{2})$ and $X^{s}(\R^{2})$ being defined by \eqref{space-Hm} and \eqref{space-Xs}, respectively. The anti-derivative operator $\p_{x}^{-1}$ is defined by \eqref{anti-derivative}.

\medskip
Let $T>0$, $1\le p \le \infty$ and $X$ be a Banach space. Then, $L^{p}(0,T; X)$ denotes the space of measurable functions $u: (0, T) \to X$ such that $\|u(t)\|_{X}$ belongs to $L^{p}(0, T)$. Also, $C([0, T]; X)$ denotes the subspace of $L^{\infty}(0, T; X)$ of all continuous functions $u: [0, T] \to X$. Moreover, $C_{b}([0, \infty); X)$ is defined as the space of all continuous and uniformly bounded functions $u: [0, \infty) \to X$. 

\medskip
Throughout this paper, $C$ denotes various positive constants, which may vary from line to line during computations. Also, it may depend on the norm of the initial data. However, we note that it does not depend on the space variable $(x, y)$ and the time variable $t$. 

\section{Linearized Problem}  

In this section, we would like to analyze the linearized problem for \eqref{KPB-s} and derive the leading term of its solution. First of all, let us re-consider the following Cauchy problem: 
\begin{align}\tag{\ref{LKPB}}
\begin{split}
& \tilde{u}_{t} + \tilde{u}_{xxx} + \e \p_{x}^{-1}\tilde{u}_{yy} -\nu \tilde{u}_{xx} = 0, \ \ (x, y) \in \R^{2}, \ t>0, \\
& \tilde{u}(x, y, 0) = u_{0}(x, y), \ \ (x, y) \in \R^{2}. 
\end{split}
\end{align}
The solution to the above Cauchy problem can be written by 
\[\tilde{u}(x, y, t)=(S(t)*u_{0})(x, y), \ \ (x, y) \in \R^{2}, \ t>0, \] 
where the integral kernel $S(x, y, t)$ is defined by 
\begin{equation}\label{DEF-S}
S(x, y, t) := \mathcal{F}^{-1} \left[ \frac{1}{2\pi} e^{ - \nu t \xi^{2} + it\left(\xi^{3} - \e \frac{\eta^{2}}{\xi} \right) } \right] (x, y). 
\end{equation}
Now, we shall analyze the decay properties and the asymptotic behavior of the solution to \eqref{LKPB}. 
We start with introducing the following $L^{\infty}$-decay estimates for the kernel $S(x, y, t)$ and the solution $\tilde{u}(x, y, t)$ to \eqref{LKPB}, given by Molinet \cite{M99}. For the proofs, see Lemma~4.3 in \cite{M99}.
\begin{prop}\label{prop.L-decay-S}
Let $l$ be a non-negative integer. Then, we have 
\begin{equation}\label{L-decay-S}
\left\| \p_{x}^{l} S(\cdot, \cdot, t) \right\|_{L^{\infty}} \le C t^{ -\frac{5}{4} -\frac{l}{2} }, \ \ t>0. 
\end{equation}
\end{prop}
\begin{prop}\label{prop.L-decay-LKPB}
Let $s\ge0$ and suppose $u_{0} \in X^{s}(\R^{2})$. If there exists a non-negative integer $j$ such that $\p_{x}^{-1}\p_{y}^{j}u_{0}\in L^{1}(\R^{2})$, 
then for all non-negative integer $l$, we have  
\begin{equation}\label{L-decay-LKPB}
\left\| \p_{x}^{l}\p_{y}^{j}(S(t)*u_{0})(\cdot, \cdot) \right\|_{L^{\infty}} \le Ct^{-\frac{7}{4}-\frac{l}{2}}, \quad t>0. 
\end{equation}
\end{prop}

Next, we shall construct the asymptotic profile for the integral kernel $S(x, y, t)$. In order to state such a result, let us consider the following auxiliary Cauchy problem:
\begin{align}\label{w-linear}
\begin{split}
& \tilde{w}_{t} + \e \p_{x}^{-1}\tilde{w}_{yy} -\nu \tilde{w}_{xx} = 0, \quad (x, y) \in \R^{2}, \ t>0,\\
& \tilde{w}(x, y, 0) = u_{0}(x, y), \quad (x, y) \in \R^{2}. 
\end{split}
\end{align}
First, we would like to show that the integral kernel of the solution to the above problem \eqref{w-linear} is given by $K(x, y, t)$ which is defined in \eqref{DEF-K}. Now, it follows from the Fourier transform that 
\begin{align}
&\mathcal{F}^{-1} \left[ \frac{1}{2\pi} e^{ - \nu t \xi^{2} - it \e \frac{\eta^{2}}{\xi} } \right] (x, y)
 = \frac{1}{(2\pi)^{2}} \int_{\R}\int_{\R} e^{ -\nu t \xi^{2} - it\e \frac{\eta^{2}}{\xi}} e^{ ix\xi + iy\eta } d\xi d\eta  \nonumber \\
& = \frac{1}{(2\pi)^{ \frac{3}{2} }} \int_{\R} e^{ -\nu t \xi^{2} } \mathcal{F}^{-1}_{\eta} \left[ e^{-it\e \frac{\eta^{2}}{\xi}} \right](y) e^{ix\xi} d\xi 
= \frac{t^{ -\frac{1}{2} }}{ 4\pi^{ \frac{3}{2} }} \int_{\R} |\xi|^{ \frac{1}{2} } e^{ -\nu t \xi^{2} } e^{ \frac{i\xi y^{2}}{4t\e} - \frac{i\pi}{4}\e \mathrm{sgn} \xi } e^{ix\xi} d\xi. \label{re-DEF-K}
\end{align}
Here, we have used the following fact: 
\begin{align}
\mathcal{F}^{-1}_{\eta} \left[ e^{-it\e \frac{\eta^{2}}{\xi}} \right](y) 
& = \frac{1}{\sqrt{2\pi}} \int_{\R}e^{-it\e \frac{\eta^{2}}{\xi}} e^{iy\eta} d\eta 
= \frac{1}{\sqrt{2\pi}} e^{ \frac{i\xi y^{2}}{4t\e} } \int_{\R} e^{ -\frac{it\e}{\xi}\left( \eta - \frac{\xi y}{2t\e} \right)^{2} } d\eta \nonumber \\
& = \frac{1}{\sqrt{2\pi}}\left( \frac{\xi \pi}{it\e} \right)^{\frac{1}{2}} e^{ \frac{i\xi y^{2}}{4t\e} }
= \sqrt{\frac{|\xi|}{2t}} e^{ \frac{i\xi y^{2}}{4t\e} - \frac{i\pi}{4}\e \mathrm{sgn} \xi }. \label{fact}
\end{align}
Therefore, by the change of valuable, it follows from \eqref{re-DEF-K}, \eqref{DEF-K} and \eqref{DEF-K*} that 
\begin{align}
& \mathcal{F}^{-1} \left[ \frac{1}{2\pi} e^{ - \nu t \xi^{2} - it \e \frac{\eta^{2}}{\xi} } \right] (x, y) \nonumber \\
& = \frac{t^{ -\frac{1}{2} }}{ 4\pi^{ \frac{3}{2} }} \int_{\R} |\xi|^{ \frac{1}{2} } e^{ -\nu t \xi^{2} } e^{ \frac{i\xi y^{2}}{4t\e} - \frac{i\pi}{4}\e \mathrm{sgn} \xi } e^{ix\xi} d\xi \nonumber\\
& = \frac{t^{ -\frac{1}{2} } }{ 4\pi^{ \frac{3}{2} }} \left( \int_{0}^{\infty} |\xi|^{ \frac{1}{2} } e^{ -\nu t \xi^{2} } e^{ \frac{i\xi y^{2}}{4t\e} - \frac{i\pi}{4}\e} e^{ix\xi} d\xi 
+ \int_{-\infty}^{0} |\xi|^{ \frac{1}{2} } e^{ -\nu t \xi^{2} } e^{ \frac{i\xi y^{2}}{4t\e} + \frac{i\pi}{4}\e} e^{ix\xi} d\xi \right) \nonumber \\
& = \frac{  t^{ -\frac{5}{4} } }{ 8\pi^{ \frac{3}{2} } \nu^{\frac{3}{4}} } \left( \int_{0}^{\infty} r^{-\frac{1}{4} }e^{ -r } e^{ \frac{iy^{2}}{4\e}\sqrt{\frac{r}{\nu}}t^{-\frac{3}{2}} - \frac{i\pi}{4}\e +ix \sqrt{\frac{r}{\nu t}}} dr 
+ \int_{0}^{\infty} r^{-\frac{1}{4} }e^{ -r } e^{ -\frac{iy^{2}}{4\e}\sqrt{\frac{r}{\nu}}t^{-\frac{3}{2}} + \frac{i\pi}{4}\e - ix \sqrt{\frac{r}{\nu t}}} dr \right)  \nonumber \\
& = \frac{  t^{ -\frac{5}{4} } }{ 8\pi^{ \frac{3}{2} } \nu^{\frac{3}{4}}} \int_{0}^{\infty} r^{-\frac{1}{4} }e^{ -r } \left( e^{ \frac{iy^{2}}{4\e}\sqrt{\frac{r}{\nu}}t^{-\frac{3}{2}} - \frac{i\pi}{4}\e +ix \sqrt{\frac{r}{\nu t}}} + e^{ -\frac{iy^{2}}{4\e}\sqrt{\frac{r}{\nu}}t^{-\frac{3}{2}} +\frac{i\pi}{4}\e -ix \sqrt{\frac{r}{\nu t}}}  \right) dr \nonumber \\
& = \frac{  t^{ -\frac{5}{4} } }{ 4\pi^{ \frac{3}{2} } \nu^{\frac{3}{4}}} \int_{0}^{\infty} r^{-\frac{1}{4} }e^{ -r } \cos \left( x \sqrt{\frac{r}{\nu t}} + \frac{y^{2}}{4\e}\sqrt{\frac{r}{\nu}}t^{-\frac{3}{2}} - \frac{\pi}{4}\e \right) dr \nonumber \\
& =  t^{ -\frac{5}{4} } K_{*} \left( xt^{-\frac{1}{2}}, yt^{-\frac{3}{4}} \right) \nonumber \\
& = K(x, y, t). \label{ap-derivation}
\end{align}
Thus, we can see that $K(x, y, t)$ is the integral kernel of the solution to the Cauchy problem \eqref{w-linear}, and then the solution $\tilde{w}(x, y, t)$ can be written by 
\[\tilde{w}(x, y, t)=(K(t)*u_{0})(x, y), \ \ (x, y) \in \R^{2}, \ t>0. \] 
Moreover, for the kernel $K(x, y, t)$ and the solution $\tilde{w}(x, y, t)$ to \eqref{w-linear}, we are able to obtain the following $L^{\infty}$-decay estimates similar to Propositions~\ref{prop.L-decay-S} and \ref{prop.L-decay-LKPB}: 
\begin{prop}\label{prop.L-decay-K}
Let $l$ be a non-negative integer. Then, we have 
\begin{equation}\label{L-decay-K}
\left\| \p_{x}^{l} K(\cdot, \cdot, t) \right\|_{L^{\infty}} \le C t^{ -\frac{5}{4} -\frac{l}{2} }, \ \ t>0. 
\end{equation}
\end{prop}
\begin{proof}
From the definitions of $K(x, y, t)$ and $K_{*}(x, y)$ (i.e. \eqref{DEF-K} and \eqref{DEF-K*}, respectively), we have 
\begin{align}
& \left|\p_{x}^{l}K(x, y, t)\right| 
= t^{-\frac{5}{4}-\frac{l}{2}} \left|\p_{x}^{l} K_{*} \left( xt^{-\frac{1}{2}}, yt^{-\frac{3}{4}} \right) \right|  \nonumber \\
& \le \frac{t^{-\frac{5}{4}-\frac{l}{2}}}{4\pi^{\frac{3}{2}} \nu^{\frac{3}{4}+\frac{l}{2}}}\int_{0}^{\infty} r^{ \frac{l}{2}-\frac{1}{4} }e^{-r} dr 
= \frac{ \Gamma \left( \frac{l}{2} + \frac{3}{4} \right) }{4\pi^{ \frac{3}{2} }\nu^{\frac{3}{4}+\frac{l}{2}}} t^{ -\frac{5}{4} -\frac{l}{2} }, \ \ (x, y) \in \R^{2}, \ t>0. \label{L-ap-est-pre}
\end{align}
This completes the proof of \eqref{L-decay-K}. 
\end{proof}
\begin{cor}\label{cor.L-decay-Lw}
Let $s\ge0$ and suppose $u_{0} \in X^{s}(\R^{2})$. If there exists a non-negative integer $j$ such that $\p_{x}^{-1}\p_{y}^{j}u_{0}\in L^{1}(\R^{2})$, 
then for all non-negative integer $l$, we have  
\begin{equation}\label{L-decay-Lw}
\left\| \p_{x}^{l}\p_{y}^{j}(K(t)*u_{0})(\cdot, \cdot) \right\|_{L^{\infty}} \le Ct^{-\frac{7}{4}-\frac{l}{2}}, \quad t>0. 
\end{equation}
\end{cor}
\begin{proof}
Since we can check that $\p_{x}^{l}\p_{y}^{j}(K(t)*u_{0})(x, y) = (\p_{x}^{l+1}K(t)*\p_{x}^{-1}\p_{y}^{j}u_{0})(x, y)$ is true through the Fourier transform, the desired estimate \eqref{L-decay-Lw} directly follows from the Young inequality. 
\end{proof}

Furthermore, we can prove that $S(x, y, t)$ is well approximated by $K(x, y, t)$ as follows:  
\begin{prop}\label{prop.L-ap-S}
Let $l$ be a non-negative integer. Then, we have 
\begin{equation}\label{L-ap-S}
\left\| \p_{x}^{l} \left( S(\cdot, \cdot, t) - K(\cdot, \cdot, t) \right) \right\|_{L^{\infty}} \le C t^{ -\frac{7}{4} -\frac{l}{2} }, \ \ t>0, 
\end{equation}
where $S(x, y, t)$ and $K(x, y, t)$ are defined by \eqref{DEF-S} and \eqref{DEF-K}, respectively. 
\end{prop}
\begin{proof}
From the mean value theorem, there exists $\theta = \theta(\xi, t)\in (0, 1)$ such that $e^{ it\xi^{3} } = 1 + it\xi^{3} e^{ i\theta t\xi^{3} }$.
Therefore, in the same way to get \eqref{re-DEF-K}, it follows from \eqref{DEF-S} and \eqref{fact} that 
\begin{align*}
S(x, y, t) 
& = \frac{1}{(2\pi)^{2}} \int_{\R}\int_{\R} e^{ -\nu t \xi^{2} + it\left(\xi^{3} - \e \frac{\eta^{2}}{\xi} \right)} e^{ ix\xi + iy\eta } d\xi d\eta  \nonumber \\
& = \frac{1}{(2\pi)^{ \frac{3}{2} }} \int_{\R} e^{ -\nu t \xi^{2} + it\xi^{3} } \mathcal{F}^{-1}_{\eta} \left[ e^{-it\e \frac{\eta^{2}}{\xi}} \right](y) e^{ix\xi} d\xi \nonumber \\
&= \frac{t^{ -\frac{1}{2} }}{ 4\pi^{ \frac{3}{2} }} \int_{\R} |\xi|^{ \frac{1}{2} } e^{ -\nu t \xi^{2} + it\xi^{3} } e^{ \frac{i\xi y^{2}}{4t\e} - \frac{i\pi}{4}\e \mathrm{sgn} \xi } e^{ix\xi} d\xi \nonumber \\
& = \frac{t^{ -\frac{1}{2} }}{ 4\pi^{ \frac{3}{2} }} \int_{\R} |\xi|^{ \frac{1}{2} } e^{ -\nu t \xi^{2} } e^{ \frac{i\xi y^{2}}{4t\e} - \frac{i\pi}{4}\e \mathrm{sgn} \xi } e^{ix\xi} d\xi + \frac{i t^{ \frac{1}{2} }}{ 4\pi^{ \frac{3}{2} }} \int_{\R} \xi^{3} |\xi|^{ \frac{1}{2} } e^{ -\nu t \xi^{2} + i\theta t\xi^{3} } e^{ \frac{i\xi y^{2}}{4t\e} - \frac{i\pi}{4}\e \mathrm{sgn} \xi } e^{ix\xi} d\xi. 
\end{align*}
Then, recalling \eqref{ap-derivation}, it says that
\begin{equation}\label{rere-DEF-S}
S(x,y,t)-K(x,y,t)=\frac{i t^{ \frac{1}{2} }}{ 4\pi^{ \frac{3}{2} }} \int_{\R} \xi^{3} |\xi|^{ \frac{1}{2} } e^{ -\nu t \xi^{2} + i\theta t\xi^{3} } e^{ \frac{i\xi y^{2}}{4t\e} - \frac{i\pi}{4}\e \mathrm{sgn} \xi } e^{ix\xi} d\xi =:R(x,y,t). 
\end{equation}
Similarly as \eqref{L-ap-est-pre}, we can evaluate the remainder term $R(x, y, t)$ as follows: 
\begin{align}
\left|\p_{x}^{l}R(x, y, t)\right| 
& \le \frac{t^{ \frac{1}{2} }}{ 4\pi^{ \frac{3}{2} }} \int_{\R} |\xi|^{ l+\frac{7}{2} } e^{ -\nu t \xi^{2} } d\xi 
= \frac{t^{ \frac{1}{2} }}{ 2\pi^{ \frac{3}{2} }} \int_{0}^{\infty} \xi^{ l+\frac{7}{2} } e^{ -\nu t \xi^{2} } d\xi  \nonumber \\
& = \frac{  t^{ -\frac{7}{4} -\frac{l}{2} } }{ 4\pi^{ \frac{3}{2} }\nu^{\frac{l}{2} + \frac{9}{4}} } \int_{0}^{\infty} r^{ \frac{l}{2}+\frac{5}{4} }e^{ -r } dr 
= \frac{ \Gamma \left( \frac{l}{2} + \frac{9}{4} \right) }{4\pi^{ \frac{3}{2} }\nu^{\frac{l}{2} + \frac{9}{4}}} t^{ -\frac{7}{4} -\frac{l}{2} }, \ \ (x, y) \in \R^{2}, \ t>0. \label{S-R-est}
\end{align}
Thus, combining \eqref{rere-DEF-S} and \eqref{S-R-est}, we finally arrive at the desired estimate \eqref{L-ap-S}. 
\end{proof}

\begin{rem}
{\rm 
From the definition \eqref{DEF-K*} of $K_{*}(x, y)$, we can easily check $\p_{x}^{l}K_{*} \in L^{\infty}(\R^{2})$ for all non-negative integer $l$. 
Therefore, by virtue of the asymptotic formula \eqref{L-ap-S}, we obtain 
\[
\left\| \p_{x}^{l}S(\cdot, \cdot, t) \right\|_{L^{\infty}} \ge \left\| \p_{x}^{l}K_{*} \right\|_{L^{\infty}}\, t^{ -\frac{5}{4} -\frac{l}{2} }-Ct^{ -\frac{7}{4}-\frac{l}{2} }, \quad t>0.
\]
This inequality means the $L^{\infty}$-estimate \eqref{L-decay-S} is optimal with respect to the time decaying order. 
}
\end{rem}

In addition to Proposition~\ref{prop.L-ap-S}, similarly as Corollary~\ref{cor.L-decay-Lw}, we can obtain the following formula:  
\begin{cor}\label{cor.L-decay-LKPB-ap}
Let $s\ge0$ and suppose $u_{0} \in X^{s}(\R^{2})$. If there exists a non-negative integer $j$ such that $\p_{x}^{-1}\p_{y}^{j}u_{0}\in L^{1}(\R^{2})$, 
then for all non-negative integer $l$, we have  
\begin{equation}\label{L-decay-LKPB-ap}
\left\| \p_{x}^{l}\p_{y}^{j}\left((S-K)(t)*u_{0}\right)(\cdot, \cdot) \right\|_{L^{\infty}} \le Ct^{-\frac{9}{4}-\frac{l}{2}}, \quad t>0, 
\end{equation}
where $S(x, y, t)$ and $K(x, y, t)$ are defined by \eqref{DEF-S} and \eqref{DEF-K}, respectively. 
\end{cor}

In the above, we have seen that the solution $\tilde{u}(x, y, t)$ to \eqref{LKPB} can be approximated by the solution $\tilde{w}(x, y, t)$ to \eqref{w-linear}. Finally in this section, we shall prove the more detailed asymptotic formula for the solution to \eqref{w-linear}. In order to give that result, we define the following new functions: 
\begin{align}
& \mathcal{K}_{j}(x, y, t) := \int_{\R} \p_{x} K(x, y-w, t)  M_{j}(w) dw, \quad (x, y)\in \R^{2}, \ \ t>0,  \label{DEF-mathK}\\
& M_{j}(y) := \int_{\R} \p_{x}^{-1}\p_{y}^{j}u_{0}(x, y) dx, \quad \p_{x}^{-1}\p_{y}^{j}u_{0}(\cdot, y)\in L^{1}(\R), \ \mathrm{a.e.} \ y\in \R, \ j \in \mathbb{N}\cup \{0\}, \label{M(y)}
\end{align}
where $K(x, y, t)$ is defined by \eqref{DEF-K}. Then, the leading term of the solution $\tilde{w}(x, y, t)$ to \eqref{w-linear} is given by $\mathcal{K}_{j}(x, y, t)$, if the initial data $u_{0}$ satisfies $\p_{x}^{-1}\p_{y}^{j}u_{0} \in L^{1}(\R^{2})$ for some $j\in \mathbb{N}\cup \{0\}$. More precisely, the following asymptotic formula can be established: 
\begin{thm}\label{thm.L-ap-LKPB}
Let $s\ge0$ and suppose $u_{0} \in X^{s}(\R^{2})$. If there exists a non-negative integer $j$ such that $\p_{x}^{-1}\p_{y}^{j}u_{0}\in L^{1}(\R^{2})$, 
then for all non-negative integer $l$, we have  
\begin{equation}\label{L-ap-LKPB}
\lim_{t \to \infty} t^{ \frac{7}{4}+\frac{l}{2} } \left\| \p_{x}^{l}\p_{y}^{j}(K(t)*u_{0})(\cdot, \cdot) - \p_{x}^{l}\mathcal{K}_{j}(\cdot, \cdot, t) \right\|_{L^{\infty}} = 0,
\end{equation}
where $K(x, y, t)$ and $\mathcal{K}_{j}(x, y, t)$ are defined by \eqref{DEF-K} and \eqref{DEF-mathK}, respectively. 
\end{thm}
\begin{proof}
Under the above assumptions, in order to show \eqref{L-ap-LKPB}, it is sufficient to prove 
\begin{equation}\label{K-ap}
\lim_{t \to \infty} t^{ \frac{7}{4}+\frac{l}{2} } \left\| (\p_{x}^{l+1}K(t)*\p_{x}^{-1}\p_{y}^{j}u_{0})(\cdot, \cdot) - \p_{x}^{l}\mathcal{K}_{j}(\cdot, \cdot, t) \right\|_{L^{\infty}} = 0.
\end{equation}
First, since $\p_{x}^{-1}\p_{y}^{j}u_{0}\in L^{1}(\R^{2})$, for any $\e_{0}>0$, there exists a constant $L=L(\e_{0})>0$ such that 
\begin{equation}\label{L1-far}
\int_{|(x, y)| \ge L} \left|\p_{x}^{-1}\p_{y}^{j}u_{0}(x, y)\right| dxdy < \e_{0}.
\end{equation}
Therefore, from \eqref{DEF-mathK} and \eqref{M(y)}, splitting the integral and applying the mean value theorem, there exists $\theta=\theta(x, z)\in (0, 1)$ such that  
\begin{align}
&(\p_{x}^{l+1}K(t)*\p_{x}^{-1}\p_{y}^{j}u_{0})(x, y) -\p_{x}^{l}\mathcal{K}_{j}(x, y, t) \nonumber \\
& = \int_{\R^{2}} \p_{x}^{l+1}K(x-z, y-w, t)\p_{z}^{-1}\p_{w}^{j}u_{0}(z, w)dzdw - \int_{\R}\p_{x}^{l+1}K(x, y-w, t) M_{j}(w) dw \nonumber \\
& = \left( \int_{|(z, w)| \le L}+\int_{|(z, w)| \ge L} \right)\p_{x}^{l+1}\left( K(x-z, y-w, t)-K(x, y-w, t) \right)\p_{z}^{-1}\p_{w}^{j}u_{0}(z, w)dzdw \nonumber \\
& = \int_{|(z, w)| \le L} (-z)\p_{x}^{l+2}K(x-\theta z, y-w, t)\p_{z}^{-1}\p_{w}^{j}u_{0}(z, w)dzdw \nonumber \\
&\ \ \ \ + \int_{|(z, w)| \ge L} \p_{x}^{l+1}\left( K(x-z, y-w, t)-K(x, y-w, t)\right)\p_{z}^{-1}\p_{w}^{j}u_{0}(z, w)dzdw \nonumber \\
& =: J_{1}(x, y, t) + J_{2}(x, y, t). \label{L-ap-split}
\end{align}

Next, we shall evaluate $J_{1}(x, y, t)$ and $J_{2}(x, y, t)$. From \eqref{L-ap-split} and Proposition~\ref{prop.L-decay-K}, we have 
\begin{align}
\left| J_{1}(x, y, t) \right| 
& \le \int_{|(z, w)| \le L} |z|\,\left|\p_{x}^{l+2}K(x-\theta z, y-w, t)\right|\,\left|\p_{z}^{-1}\p_{w}^{j}u_{0}(z, w)\right|dzdw \nonumber \\
& \le Ct^{ -\frac{5}{4} -\frac{l+2}{2} } \int_{|(z, w)| \le L} |(z, w)|\,\left|\p_{z}^{-1}\p_{w}^{j}u_{0}(z, w)\right|dzdw \nonumber \\
& \le CL \left\| \p_{x}^{-1}\p_{y}^{j}u_{0} \right\|_{L^1}t^{ -\frac{9}{4} -\frac{l}{2} }, \ \ (x, y) \in \R^{2}, \ t>0. \label{L-ap-est-J1}
\end{align}
On the other hand, from \eqref{L-ap-split}, the triangle inequality, Proposition~\ref{prop.L-decay-K} and \eqref{L1-far}, we get
\begin{align}
\left| J_{2}(x, y, t) \right| 
& \le \int_{|(z, w)| \ge L} \left( \left|\p_{x}^{l+1}K(x-z, y-w, t)\right| + \left|\p_{x}^{l+1}K(x, y-w, t)\right| \right) \left|\p_{z}^{-1}\p_{w}^{j}u_{0}(z, w)\right|dzdw \nonumber \\
& \le Ct^{ -\frac{5}{4} -\frac{l+1}{2} } \int_{|(z, w)| \ge L} \left|\p_{z}^{-1}\p_{w}^{j}u_{0}(z, w)\right|dzdw \nonumber \\
&\le C\e_{0}t^{ -\frac{7}{4} -\frac{l}{2} }, \ \ (x, y) \in \R^{2}, \ t>0. \label{L-ap-est-J2}
\end{align}

Summarizing up \eqref{L-ap-split}, \eqref{L-ap-est-J1} and \eqref{L-ap-est-J2}, we obtain 
\[
\left\| (\p_{x}^{l+1}K(t)*\p_{x}^{-1}\p_{y}^{j}u_{0})(\cdot, \cdot) - \p_{x}^{l}\mathcal{K}_{j}(\cdot, \cdot, t) \right\|_{L^{\infty}} 
\le Ct^{ -\frac{9}{4}-\frac{l}{2} } + C\e_{0} t^{ -\frac{7}{4}-\frac{l}{2} }, \quad t>0.
\] 
Therefore, we finally arrive at
\[
\limsup_{t \to \infty} t^{ \frac{7}{4}+\frac{l}{2} } \left\| (\p_{x}^{l+1}K(t)*\p_{x}^{-1}\p_{y}^{j}u_{0})(\cdot, \cdot) - \p_{x}^{l}\mathcal{K}_{j}(\cdot, \cdot, t) \right\|_{L^{\infty}} \le C\e_{0}.
\] 
Thus, we get \eqref{K-ap}, because $\e_{0}>0$ can be chosen arbitrarily small. This completes the proof. 
\end{proof}

\section{Proofs of the Main Results}  

In this section, we shall prove our main results Theorems~\ref{thm.u-asymp} and \ref{thm.u-est-lower}, i.e. we discuss the asymptotic behavior of the solution $u(x, y, t)$ to \eqref{KPB-s}. This section is divided into three subsections below. 

\subsection{Preliminaries}  

In this subsection, we shall prepare a couple of important propositions and lemmas to prove our main results. 
We start with introducing some basic results given by Molinet~\cite{M99}. First one is for the global well-posedness of \eqref{KPB-s} (for the proof, see Theorem~3.2 in \cite{M99}). 
\begin{thm}\label{thm.u-SDGE}
Let $p\ge1$ be an integer and $s>2$. Suppose that $u_{0} \in X^{s}(\R^{2})$ and $B(u_{0}):=\left\|u_{0}\right\|_{H^{1}}+\left\|\p_{x}^{2}u_{0}\right\|_{L^{2}}+\left\|\p^{-1}_{x}\p_{y}u_{0}\right\|_{L^{2}}$ is sufficiently small. Then \eqref{KPB-s} has a unique global solution 
\begin{equation}\label{u-SDGE}
u \in C_{b}([0, \infty); X^{s}(\R^{2})) \cap L^{2}(0, \infty; H^{s}(\R^{2})),
\end{equation}
and the mapping $u_{0} \mapsto u$ is continuous from $X^{s}(\R^{2})$ into $C([0, \infty); X^{s}(\R^{2}))$. 
Moreover, the solution satisfies the following a priori estimate: 
\begin{equation}\label{apriori-est}
\left\|u(\cdot, \cdot, t)\right\|_{X^{s}}^{2}+\int_{0}^{t}\left\|u_{x}(\cdot, \cdot, \tau)\right\|_{X^{s}}^{2}d\tau \le C_{s}\left(B(u_{0}), \|u_{0}\|_{X^{s}}\right), \ \ t\ge 0,  
\end{equation}
where $C_{s}\left(B(u_{0}), \|u_{0}\|_{X^{s}}\right)$ is a certain positive constant depending on $B(u_{0})$ and $\|u_{0}\|_{X^{s}}$, which goes to $0$ as $u_{0} \to 0 $ in $X^{s}(\R^{2})$. Furthermore, for all positive integer $l$, the solution satisfies 
\[
\p_{x}^{l}u \in C_{b}([1, \infty); H^{s}(\R^{2})) \cap L^{2}(1, \infty; H^{s}(\R^{2})),
\] 
and the mapping $\p_{x}^{l}u(1) \mapsto \p_{x}^{l}u$ is continuous from $H^{s}(\R^{2})$ into $C([1, \infty); H^{s}(\R^{2}))$. 
\end{thm}

Secondly, we shall introduce several decay estimates for $u(x, y, t)$ to \eqref{KPB-s}. The following three propositions play an important role in the proofs of our main results (for the proofs of the propositions below, see Theorem~5.1, Corollary~5.1 and Theorem~5.2 in \cite{M99}, respectively). 
\begin{prop}\label{prop.u-decay-L2}
Let $p\ge1$ be an integer and $s>2$. Suppose that $u_{0} \in X^{s}(\R^{2})$, $\p_{x}^{-1}u_{0} \in L^{1}(\R^{2})$ and $B(u_{0})$ is sufficiently small. If for some integer $m\ge1$, one has $\p_{x}^{-1}\p_{y}^{m}u_{0} \in L^{1}(\R^{2})$, then the solution $u(x,y,t)$ to \eqref{KPB-s} satisfies 
\begin{align}
&\left\|\p_{x}^{l}\p_{y}^{j}u(\cdot, \cdot, t)\right\|_{L^{2}} \le C(1+t)^{-\frac{3}{4}-\frac{l}{2}}, \ \ t\ge0, \label{u-decay-L2}
\end{align}
for all $l = 0, 1, 2$ and $j = 0, 1, \cdots, \min\{[s], m\}-1$. 
\end{prop}
\begin{prop}\label{prop.u-decay-L2-high}
Let $p\ge1$ be an integer and $s>2$. Suppose that $u_{0} \in X^{s}(\R^{2})$, $\p_{x}^{-1}u_{0} \in L^{1}(\R^{2})$ and $B(u_{0})$ is sufficiently small. If for some integer $m\ge2$, one has $\p_{x}^{-1}\p_{y}^{m}u_{0} \in L^{1}(\R^{2})$, then the solution $u(x,y,t)$ to \eqref{KPB-s} satisfies 
\begin{equation}
\left\|\p_{x}^{l}\p_{y}^{j}u(\cdot, \cdot, t)\right\|_{L^{2}} \le Ct^{-\frac{3}{4}-\frac{l}{2}}, \ \ t\ge1, \label{u-decay-L2-high}
\end{equation}
for all positive integer $l$ and $j = 0, 1, \cdots, \min\{[s], m\}-1$. 
\end{prop}
\begin{prop}\label{prop.u-decay-Linf}
Let $p\ge2$ be an integer and $s\ge3$. Suppose that $u_{0} \in X^{s}(\R^{2})$, $\p_{x}^{-1}u_{0} \in L^{1}(\R^{2})$, $\p_{x}^{-1}\p_{y}^{m}u_{0} \in L^{1}(\R^{2})$ for some integer $m\ge3$, and $B(u_{0})$ is sufficiently small. If $\p_{x}^{-1}\p_{y}^{j}u_{0} \in L^{1}(\R^{2})$ with $j \in \left\{0, 1, \cdots, \min\{[s], m\}-3\right\}$, then the solution $u(x,y,t)$ to \eqref{KPB-s} satisfies 
\begin{equation}
\left\|\p_{x}^{l}\p_{y}^{j}u(\cdot, \cdot, t)\right\|_{L^{\infty}} \le Ct^{-\frac{7}{4}-\frac{l}{2}}, \ \ t\ge1, \label{u-decay-Linf}
\end{equation}
for all non-negative integer $l$. 
\end{prop}

Next, we would like to prepare the following lemma to treat the Duhamel term of the integral equations \eqref{integral-eq} below and \eqref{w-sol}, associated with the Cauchy problems \eqref{KPB-s} and \eqref{w-asymp}. 
\begin{lem}\label{lem.Duhamel-est}
Let $t>a>0$, $f \in C((0, \infty); L^{1}(\R^{2}))$ and $\p_{y}^{2}f \in C((0, \infty); L^{1}(\R^{2}))$. Then, we have 
\begin{align}
&\left\| \int_{a}^{t} S(t-\tau)*f(\tau) d\tau \right\|_{L^{\infty}}
\le C\sqrt{t} \sup_{a \le \tau \le t}\left( \left\|f(\cdot, \cdot, \tau)\right\|_{L^{1}}+\left\|\p_{y}^{2}f(\cdot, \cdot, \tau)\right\|_{L^{1}} \right), \label{Duhamel-S-est} \\
&\left\| \int_{a}^{t} K(t-\tau)*f(\tau) d\tau \right\|_{L^{\infty}}
\le C\sqrt{t} \sup_{a \le \tau \le t}\left( \left\|f(\cdot, \cdot, \tau)\right\|_{L^{1}}+\left\|\p_{y}^{2}f(\cdot, \cdot, \tau)\right\|_{L^{1}} \right), \label{Duhamel-K-est}
\end{align}
where $S(x, y, t)$ and $K(x, y, t)$ are defined by \eqref{DEF-S} and \eqref{DEF-K}, respectively. 
\end{lem}
\begin{proof}
We shall only prove \eqref{Duhamel-S-est} because \eqref{Duhamel-K-est} can be shown in the same way. It follows from the Fourier transform and the definition of $S(x, y, t)$ (i.e. \eqref{DEF-S}) that 
\begin{align*}
&\left\| \int_{a}^{t} S(t-\tau)*f(\tau) d\tau \right\|_{L^{\infty}}
 = \sup_{(x, y)\in \R^{2}}\left| \int_{a}^{t}\int_{\R^{2}}\hat{S}(\xi, \eta, t-\tau) e^{ix\xi+iy\eta} \hat{f}(\xi, \eta, \tau) d\xi d\eta d\tau \right| \\
& \le \int_{a}^{t}\int_{\R^{2}}\left|\hat{S}(\xi, \eta, t-\tau)\right|\left\{1-(i\eta)^{2}\right\}^{-1}
\left|\left\{1-(i\eta)^{2}\right\}\hat{f}(\xi, \eta, \tau) \right| d\xi d\eta d\tau\\
& \le C\int_{a}^{t}\left( \left\|f(\cdot, \cdot, \tau)\right\|_{L^{1}}+\left\|\p_{y}^{2}f(\cdot, \cdot, \tau)\right\|_{L^{1}} \right)\left( \int_{\R}\exp\left(-\nu(t-\tau)\xi^{2}\right)d\xi \right) \left( \int_{\R}\frac{1}{1+\eta^{2}} d\eta \right) d\tau \\
&\le C\sup_{a\le \tau \le t}\left( \left\|f(\cdot, \cdot, \tau)\right\|_{L^{1}}+\left\|\p_{y}^{2}f(\cdot, \cdot, \tau)\right\|_{L^{1}} \right)\int_{a}^{t} \sqrt{\frac{\pi}{\nu(t-\tau)}} d\tau \\
&\le C\sqrt{t} \sup_{a\le \tau \le t}\left( \left\|f(\cdot, \cdot, \tau)\right\|_{L^{1}}+\left\|\p_{y}^{2}f(\cdot, \cdot, \tau)\right\|_{L^{1}} \right).
\end{align*}
This completes the proof of \eqref{Duhamel-S-est}. 
\end{proof}

In order to handle the Duhamel term of the integral equations, in addition to the above lemma, 
we need to introduce the following decay estimates for the nonlinear term: 
\begin{lem}\label{lem.u^p+1-est}
Under the same assumptions in Proposition~\ref{prop.u-decay-Linf}, the solution $u(x,y,t)$ to \eqref{KPB-s} satisfies 
\begin{align}
&\left\|\p_{y}^{j}(u^{p+1}(\cdot, \cdot, \tau))\right\|_{L^{1}} \le C_{s}\left(B(u_{0}), \|u_{0}\|_{X^{s}}\right)^{p-1}(1+\tau)^{-\frac{3}{2}}, \ \ \tau \ge0, \label{u^p+1-est-1} \\
&\left\|\p_{x}^{l+1}\p_{y}^{j}(u^{p+1}(\cdot, \cdot, \tau))\right\|_{L^{1}}+\left\|\p_{x}^{l+1}\p_{y}^{j+2}(u^{p+1}(\cdot, \cdot, \tau))\right\|_{L^{1}} 
\le C\tau^{-\frac{l+1}{2}-\frac{3}{4}(p+1)}, \ \ \tau \ge1, \label{u^p+1-est-2} 
\end{align}
for all non-negative integer $l$. 
\end{lem}
\begin{proof}
First, we shall prove \eqref{u^p+1-est-1}. From the Leibniz rule and the Schwarz inequality, we have 
\begin{align}
&\left\|\p_{y}^{j}(u^{p+1}(\cdot, \cdot, \tau))\right\|_{L^{1}} \nonumber  \\
&\le C \sum_{\substack{q_{1}+\cdots+q_{p+1}=j; \\ q_{1} \ge \cdots \ge q_{p+1}}} \left(\prod_{k=3}^{p+1} \left\| \p_{y}^{q_{k}}u(\cdot, \cdot, \tau) \right\|_{L^{\infty}} \right) 
\left\| \p_{y}^{q_{1}}u(\cdot, \cdot, \tau) \right\|_{L^{2}} \left\| \p_{y}^{q_{2}}u(\cdot, \cdot, \tau) \right\|_{L^{2}}. \label{u^p+1-est-pre1}
\end{align}
If $q_{k}\le j \le [s]-3$, from \eqref{u-SDGE}, \eqref{space-Xs} and the Sobolev embedding theorem, we can see that the solution satisfies $\p_{y}^{q_{k}}u \in C_{b}([0, \infty); L^{\infty}(\R^{2}))$. More precisely, it follows from \eqref{apriori-est} that 
\begin{align}
\left\| \p_{y}^{q_{k}}u(\cdot, \cdot, \tau) \right\|_{L^{\infty}}
&\le C\left\| \p_{y}^{q_{k}}u(\cdot, \cdot, \tau) \right\|_{H^{s-[s]+3}} 
\le C\left\| \p_{y}^{q_{k}}u(\cdot, \cdot, \tau) \right\|_{X^{s-[s]+3}} \nonumber \\
&\le C\left\| u(\cdot, \cdot, \tau) \right\|_{X^{s}} 
\le C_{s}\left(B(u_{0}), \|u_{0}\|_{X^{s}}\right), \ \ \tau \ge0. \label{embedding}
\end{align}
Therefore, by virtue of \eqref{u^p+1-est-pre1}, \eqref{embedding} and Proposition~\ref{prop.u-decay-L2}, we immediately obtain \eqref{u^p+1-est-1}. 

Next, let us prove \eqref{u^p+1-est-2}. Since the $y$-derivative do not improve the decay rate of the solution (cf.~Proposition~\ref{prop.u-decay-Linf}), we only consider the estimate for the second term of the left hand side of \eqref{u^p+1-est-2}. In the same way to get \eqref{u^p+1-est-pre1}, we have 
\begin{align}
&\left\|\p_{x}^{l+1}\p_{y}^{j+2}(u^{p+1}(\cdot, \cdot, \tau))\right\|_{L^{1}} \nonumber  \\
&\le C \sum_{\substack{q_{1}+\cdots+q_{p+1}=j+2; \\ q_{1} \ge \cdots \ge q_{p+1}, \\ r_{1}+\cdots+r_{p+1}=l+1}} \left(\prod_{k=3}^{p+1} \left\| \p_{x}^{r_{k}}\p_{y}^{q_{k}}u(\cdot, \cdot, \tau) \right\|_{L^{\infty}} \right) 
\left\| \p_{x}^{r_{1}}\p_{y}^{q_{1}}u(\cdot, \cdot, \tau) \right\|_{L^{2}} \left\| \p_{x}^{r_{2}}\p_{y}^{q_{2}}u(\cdot, \cdot, \tau) \right\|_{L^{2}}. \label{u^p+1-est-pre2}
\end{align}
For $k\ge3$, since $q_{k}\le j \le \min\{[s], m\}-3$, by using an inequality (cf.~\cite{BIN78})
\begin{equation*}
\left\|f\right\|_{L^{\infty}}\le C\left( \left\|f\right\|_{L^{2}}+\left\|f_{y}\right\|_{L^{2}}+\left\|f_{xx}\right\|_{L^{2}} \right), \quad f\in H^{2}(\R^{2})
\end{equation*}
 and Proposition~\ref{prop.u-decay-L2-high}, we obtain 
\begin{align}
\left\| \p_{x}^{r_{k}}\p_{y}^{q_{k}}u(\cdot, \cdot, \tau) \right\|_{L^{\infty}} 
&\le C\left( \left\| \p_{x}^{r_{k}}\p_{y}^{q_{k}}u(\cdot, \cdot, \tau) \right\|_{L^{2}} + \left\| \p_{x}^{r_{k}}\p_{y}^{q_{k}+1}u(\cdot, \cdot, \tau) \right\|_{L^{2}} + \left\| \p_{x}^{r_{k}+2}\p_{y}^{q_{k}}u(\cdot, \cdot, \tau) \right\|_{L^{2}} \right) \nonumber \\
&\le C\tau^{-\frac{r_{k}}{2}-\frac{3}{4}}, \ \ \tau \ge1. \label{u^p+1-est-pre3}
\end{align}
For $k=1, 2$, since $q_{k}\le j+2 \le \min\{[s], m\}-1$, it follows from Propositions~\ref{prop.u-decay-L2} and \ref{prop.u-decay-L2-high} that  
\begin{equation}\label{u^p+1-est-pre4}
\left\| \p_{x}^{r_{1}}\p_{y}^{q_{1}}u(\cdot, \cdot, \tau) \right\|_{L^{2}}\left\| \p_{x}^{r_{2}}\p_{y}^{q_{2}}u(\cdot, \cdot, \tau) \right\|_{L^{2}} \le C\tau^{-\frac{r_{1}+r_{2}}{2}-\frac{3}{2}}, \ \ \tau \ge1.
\end{equation}
Therefore, combining \eqref{u^p+1-est-pre2}, \eqref{u^p+1-est-pre3} and \eqref{u^p+1-est-pre4}, we arrive at 
\begin{equation*}
\left\|\p_{x}^{l+1}\p_{y}^{j}(u^{p+1}(\cdot, \cdot, \tau))\right\|_{L^{1}}+\left\|\p_{x}^{l+1}\p_{y}^{j+2}(u^{p+1}(\cdot, \cdot, \tau))\right\|_{L^{1}} 
\le C\tau^{-\frac{l+1}{2}-\frac{3}{4}(p+1)}, \ \ \tau \ge1.
\end{equation*}
Thus, we obtain \eqref{u^p+1-est-2}. This completes the proof. 
\end{proof}

\begin{cor}\label{M-uniform}
Under the same assumptions in Proposition~\ref{prop.u-decay-Linf}, the solution $u(x,y,t)$ to \eqref{KPB-s} satisfies 
\begin{equation}\label{u^p+1-uniform-est}
\int_{0}^{\infty}\int_{\R^{2}}\left|\p_{y}^{j}\left(u^{p+1}(x, y, t)\right)\right|dxdydt\le C_{s}\left(B(u_{0}), \|u_{0}\|_{X^{s}}\right)^{p-1}.
\end{equation}
\end{cor}
\begin{proof}
It immediately follows from \eqref{u^p+1-est-1} that 
\begin{align*}
&\int_{0}^{\infty}\int_{\R^{2}}\left|\p_{y}^{j}\left(u^{p+1}(x, y, t)\right)\right|dxdydt
= \int_{0}^{\infty}\left\|\p_{y}^{j}(u^{p+1}(\cdot, \cdot, \tau))\right\|_{L^{1}} d\tau  \\
&\le C_{s}\left(B(u_{0}), \|u_{0}\|_{X^{s}}\right)^{p-1}\int_{0}^{\infty}(1+\tau)^{-\frac{3}{2}} d\tau 
\le C_{s}\left(B(u_{0}), \|u_{0}\|_{X^{s}}\right)^{p-1}. 
\end{align*}
\end{proof}
\begin{rem}\label{rem-M}
{\rm
From \eqref{u^p+1-uniform-est}, we can check that $\mathcal{M}_{j}$ defined by \eqref{DEF-mathM} is finite. 
Actually, $\mathcal{M}_{j}$ can be evaluated a priori as follows: 
\[
\left|\mathcal{M}_{j}\right| \le \left\|\p_{x}^{-1}\p_{y}^{j}u_{0}\right\|_{L^{1}} + C_{s}\left(B(u_{0}), \|u_{0}\|_{X^{s}}\right)^{p-1}. 
\]
Moreover, under the smallness assumptions on the initial data $u_{0}$, we can lead $\mathcal{M}_{j} \neq 0$. 
Indeed, if we take the initial data $u_{0}$ to satisfy the following conditions: 
\begin{equation}\label{data-condition}
C_{s}\left(B(u_{0}), \|u_{0}\|_{X^{s}}\right)< \left\{\left|\int_{\R^{2}}\p_{x}^{-1}\p_{y}^{j}u_{0}(x, y)dxdy\right|\right\}^{\frac{1}{p-1}}, \quad 
\int_{\R^{2}}\p_{x}^{-1}\p_{y}^{j}u_{0}(x, y)dxdy\neq0, 
\end{equation}
then we are able to see that 
\begin{align*}
\left|\mathcal{M}_{j}\right| 
&\ge \left|\int_{\R^{2}}\p_{x}^{-1}\p_{y}^{j}u_{0}(x, y)dxdy\right|-\frac{1}{p+1}\int_{0}^{\infty}\int_{\R^{2}}\left|\p_{y}^{j}\left(u^{p+1}(x, y, t)\right)\right|dxdydt \\
&\ge \left|\int_{\R^{2}}\p_{x}^{-1}\p_{y}^{j}u_{0}(x, y)dxdy\right|-C_{s}\left(B(u_{0}), \|u_{0}\|_{X^{s}}\right)^{p-1}>0.
\end{align*}
}
\end{rem}
\begin{rem}\label{rem-M2}
{\rm 
If $u_0$ satisfies 
the second condition of 
{\rm (\ref{data-condition})}, 
then we can find the function $v_0$ which satisfies 
both of two conditions in 
{\rm (\ref{data-condition})} as follows. 

Let $0<\lambda <1$ and put
\[
u_0^{\lambda}(x,y):=\lambda^{\alpha}u_0\left(\lambda^{\beta}x,\lambda^{\gamma}y \right),
\]
where $\alpha$, $\beta$, $\gamma>0$ will be chosen later. 
Let $|\partial_x|^s$ denotes the Fourier multiplier 
defined by
\[
\left( |\partial_x|^s f \right)(x,y)
:=\mathcal{F}^{-1}\left[|\xi|^s\widehat{f}(\xi,\eta)\right](x,y). 
\]
By the simple calculation, we can see that
\[
\begin{split}
\left\||\partial_x|^su^{\lambda}_0\right\|_{L^2}
&=\lambda^{\alpha +\left(s-\frac{1}{2}\right)\beta-\frac{\gamma}{2}}
\left\||\partial_x|^su_0\right\|_{L^2},\\
\left\||\partial_x|^s\partial_x^{-1}u^{\lambda}_0\right\|_{L^2}
&=\lambda^{\alpha +\left(s-\frac{3}{2}\right)\beta-\frac{\gamma}{2}}
\left\||\partial_x|^s\partial_x^{-1}u_0\right\|_{L^2}
\end{split}
\]
for any $s\ge 0$ and
\[
\left\|\partial_x^{-1}\partial_yu_0^{\lambda}\right\|_{L^2}
=\lambda^{\alpha -\frac{3}{2}\beta+\frac{\gamma}{2}}
\left\|\partial_x^{-1}\partial_yu_0\right\|_{L^2}. 
\]
These facts imply that
\[
\left\|u_0^{\lambda}\right\|_{X^s}
\le 
\left(\lambda^{\alpha-\frac{\beta}{2}-\frac{\gamma}{2}}
+\lambda^{\alpha+\left(s-\frac{1}{2}\right)\beta -\frac{\gamma}{2}}
+\lambda^{\alpha-\frac{3}{2}\beta -\frac{\gamma}{2}}
+\lambda^{\alpha+\left(s-\frac{3}{2}\right)\beta -\frac{\gamma}{2}}
\right)
\left\|u_0\right\|_{X^s}
\]
and
\[
B\left(u_0^{\lambda}\right)
\le 
\left(\lambda^{\alpha-\frac{\beta}{2}-\frac{\gamma}{2}}
+\lambda^{\alpha+\frac{\beta}{2}-\frac{\gamma}{2}}
+\lambda^{\alpha+\frac{3}{2}\beta -\frac{\gamma}{2}}
+\lambda^{\alpha-\frac{3}{2}\beta +\frac{\gamma}{2}}
\right)B(u_0). 
\]
On the other hand, we have
\[
\int_{\R^2}\partial_x^{-1}\partial_y^ju_0^{\lambda}(x,y)dxdy
=\lambda^{\alpha -2\beta +(j-1)\gamma}
\int_{\R^2}\partial_x^{-1}\partial_y^ju_0(x,y)dxdy
\]
for any $j\in \mathbb{N}\cup \{0\}$. 
If we choose $\alpha$, $\beta$, $\gamma>0$ as
\[
\alpha =2\beta -(j-1)\gamma, \quad
\beta =2(j+1)\gamma, 
\]
then we obtain
\[
\begin{split}
\left\|u_0^{\lambda}\right\|_{X^s}
&\le 4\lambda^{\alpha -\frac{3}{2}\beta-\frac{\gamma}{2}}
\left\|u_0\right\|_{X^s}
=4\lambda^{\frac{3}{2}\gamma}\|u_0\|_{X^s},\\
B\left(u_0^{\lambda}\right)
&\le 4\lambda^{\alpha -\frac{3}{2}\beta+\frac{\gamma}{2}}
B(u_0)=4\lambda^{\frac{5}{2}\gamma}B(u_0)
\end{split}
\]
and
\[
\int_{\R^2}\partial_x^{-1}\partial_y^ju_0^{\lambda}(x,y)dxdy
=
\int_{\R^2}\partial_x^{-1}\partial_y^ju_0(x,y)dxdy\ne 0. 
\]
Therefore, 
$\left\|u_0^{\lambda}\right\|_{X^s}$ and $B\left(u_0^{\lambda}\right)$ 
become small if we choose $\lambda >0$ small enough. 
Namely, there exists $\Lambda_j >0$ 
such that for any $v_0\in \left\{u_0^{\lambda}|\ 0<\lambda <\Lambda_j\right\}$ satisfies {\rm (\ref{data-condition})}.
}
\end{rem}

\subsection{Proof of Theorem~\ref{thm.u-asymp}}  

In this subsection, we shall prove Theorem~\ref{thm.u-asymp}. 
To discuss the asymptotic behavior of the solution $u(x, y, t)$, let us rewrite the Cauchy problem \eqref{KPB-s} to the following integral equation: 
\begin{equation}\label{integral-eq}
u(t) = S(t)*u_{0} -\frac{1}{p+1}\int_{0}^{t} \p_{x}S(t-\tau)*\left(u^{p+1}(\tau)\right) d\tau. 
\end{equation}
Since we have already obtained the approximate formula for the linear part of \eqref{integral-eq} (see Corollary~\ref{cor.L-decay-LKPB-ap} in Section~2), it is sufficient to consider only the Duhamel term. Now, we set 
\begin{align}
D_{S}(x, y, t) &:= -\frac{1}{p+1}\int_{0}^{t} \p_{x}S(t-\tau)*\left(u^{p+1}(\tau)\right) d\tau, \label{Duhamel}\\
D_{K}(x, y, t) &:= -\frac{1}{p+1}\int_{0}^{t} \p_{x}K(t-\tau)*\left(u^{p+1}(\tau)\right) d\tau. \label{Duhamel-main}
\end{align}
Then, $D_{S}(x, y, t)$ is well approximated by $D_{K}(x, y, t)$. Actually, we have the following result: 
\begin{prop}\label{prop.Duhamel-asymp}
Under the same assumptions in Theorem~\ref{thm.u-asymp}, we have the following estimate: 
\begin{equation}
\left\|\p_{x}^{l}\p_{y}^{j}\left(D_{S}(\cdot, \cdot, t)-D_{K}(\cdot, \cdot, t)\right)\right\|_{L^{\infty}} \le C t^{-\frac{9}{4}-\frac{l}{2}}, \ \ t\ge2, \label{Duhamel-asymp}
\end{equation}
for all non-negative integer $l$. 
\end{prop}
\begin{proof}
First, we split the $\tau$-integral in \eqref{Duhamel} and \eqref{Duhamel-main} as follows: 
\begin{align}
&\p_{x}^{l}\p_{y}^{j}(D_{S}(x, y, t)-D_{K}(x, y, t)) \nonumber \\
&=-\frac{1}{p+1}\int_{0}^{t/2} \p_{x}^{l+1}(S-K)(t-\tau)*\left(\p_{y}^{j}\left(u^{p+1}(\tau)\right)\right) d\tau \nonumber \\
&\ \ \ \, -\frac{1}{p+1}\int_{t/2}^{t} (S-K)(t-\tau)*\left(\p_{x}^{l+1}\p_{y}^{j}\left(u^{p+1}(\tau)\right)\right) d\tau
=:N_{1}(x, y, t)+N_{2}(x, y, t). \label{D-mathD-first}
\end{align}
Then, it follows from the Young inequality, Proposition~\ref{prop.L-ap-S} and \eqref{u^p+1-est-1} that 
\begin{align}
\left\|N_{1}(\cdot, \cdot, t)\right\|_{L^{\infty}}
&\le \frac{1}{p+1}\int_{0}^{t/2}\left\|\p_{x}^{l+1}(S-K)(\cdot, \cdot, t-\tau)\right\|_{L^{\infty}}\left\|\p_{y}^{j}(u^{p+1}(\cdot, \cdot, \tau))\right\|_{L^{1}}d\tau \nonumber \\
&\le C \int_{0}^{t/2}(t-\tau)^{-\frac{7}{4}-\frac{l+1}{2}}(1+\tau)^{-\frac{3}{2}}d\tau 
\le  C t^{-\frac{9}{4}-\frac{l}{2}}, \ \ t>0. \label{N1-est}
\end{align}
On the other hand, applying Lemma~\ref{lem.Duhamel-est} and \eqref{u^p+1-est-2}, we have
\begin{align}
\left\|N_{2}(\cdot, \cdot, t)\right\|_{L^{\infty}}
&\le \frac{1}{p+1}\left\| \int_{t/2}^{t} S(t-\tau)*\left(\p_{x}^{l+1}\p_{y}^{j}\left(u^{p+1}(\tau)\right)\right) d\tau \right\|_{L^{\infty}} \nonumber \\
&\ \ \ \ + \frac{1}{p+1}\left\| \int_{t/2}^{t} K(t-\tau)*\left(\p_{x}^{l+1}\p_{y}^{j}\left(u^{p+1}(\tau)\right)\right) d\tau \right\|_{L^{\infty}} \nonumber\\
&\le C\sqrt{t} \sup_{t/2\le \tau \le t}\left( \left\|\p_{x}^{l+1}\p_{y}^{j}(u^{p+1}(\cdot, \cdot, \tau))\right\|_{L^{1}}+\left\|\p_{x}^{l+1}\p_{y}^{j+2}(u^{p+1}(\cdot, \cdot, \tau))\right\|_{L^{1}} \right) \nonumber\\
&\le C\sqrt{t} \sup_{t/2\le \tau \le t} \tau^{-\frac{l+1}{2}-\frac{3}{4}(p+1)} 
\le Ct^{-\frac{3}{4}(p+1)-\frac{l}{2}}, \ \ t\ge2. \label{N2-est}
\end{align}
Here, we note that $p\ge2$ is equivalent to $\frac{3}{4}(p+1)\ge \frac{9}{4}$. Therefore, combining \eqref{D-mathD-first}, \eqref{N1-est} and \eqref{N2-est}, we obtain the desired estimate \eqref{Duhamel-asymp}. This completes the proof. 
\end{proof}

\begin{proof}[\rm{\bf{End of the Proof of Theorem~\ref{thm.u-asymp}}}]
It follows from \eqref{integral-eq}, \eqref{Duhamel}, \eqref{w-sol} and \eqref{Duhamel-main} that 
\[
u(x, y, t)-w(x, y, t)=\left((S-K)(t)*u_{0}\right)(x, y)+D_{S}(x, y, t)-D_{K}(x, y, t). 
\]
Therefore, by virtue of Corollary~\ref{cor.L-decay-LKPB-ap} and Proposition~\ref{prop.Duhamel-asymp}, we immediately obtain 
\begin{align*}
&\left\|\p_{x}^{l}\p_{y}^{j}\left(u(\cdot, \cdot, t)-w(\cdot, \cdot, t)\right)\right\|_{L^{\infty}} \\
&\le \left\| \p_{x}^{l}\p_{y}^{j}\left((S-K)(t)*u_{0}\right)(\cdot, \cdot) \right\|_{L^{\infty}}+\left\|\p_{x}^{l}\p_{y}^{j}\left(D_{S}(\cdot, \cdot, t)-D_{K}(\cdot, \cdot, t)\right)\right\|_{L^{\infty}}
\le Ct^{-\frac{9}{4}-\frac{l}{2}}, \ \ t\ge2. 
\end{align*}
Thus, we are able to see that the desired estimate \eqref{u-asymp} is true. 
\end{proof}

\subsection{Proof of Theorem~\ref{thm.u-est-lower}}  

Finally in this section, we would like to derive the lower bound of the solution $u(x, y, t)$ to \eqref{KPB-s}. 
By virtue of Theorem~\ref{thm.u-asymp}, we can see that the solution $u(x, y, t)$ is well approximated by $w(x, y, t)$ given by \eqref{w-sol}. 
In addition, it follows from Theorem~\ref{thm.L-ap-LKPB} that the linear part of \eqref{w-sol} converges to $\mathcal{K}_{j}(x, y, t)$ defined by \eqref{DEF-mathK}. Therefore, in order to prove Theorem~\ref{thm.u-est-lower}, it is sufficient to evaluate $\p_{x}^{l}\mathcal{K}_{j}(x, y, t)+\p_{x}^{l}\p_{y}^{j}D_{K}(x, y, t)$ from below. Actually, we can prove the following estimate: 
\begin{prop}\label{prop.w-main-lower}
Under the same assumptions in Theorem~\ref{thm.u-asymp}, there exist positive constants $C_{*}>0$ and $T_{*}>0$ such that the following estimate
\begin{equation}\label{w-main-lower}
\left\|\p_{x}^{l}\mathcal{K}_{j}(\cdot, \cdot, t)+\p_{x}^{l}\p_{y}^{j}D_{K}(\cdot, \cdot, t)\right\|_{L^{\infty}} \ge C_{*}\left|\mathcal{M}_{j}\right|t^{-\frac{7}{4}-\frac{l}{2}}, \ \ t \ge T_{*} 
\end{equation}
holds for all non-negative integer $l$, where $\mathcal{K}_{j}(x, y, t)$, $D_{K}(x, y, t)$ and $\mathcal{M}_{j}$ are defined by \eqref{DEF-mathK}, \eqref{Duhamel-main} and \eqref{DEF-mathM}, respectively. 
\end{prop}

To prove Proposition~\ref{prop.w-main-lower}, we start with analyzing $\p_{x}^{l}\p_{y}^{j}D_{K}(x, y, t)$ in details. 
Now, we take $0<\delta<1$ and then split the integral of the right hand side of $\p_{x}^{l}\p_{y}^{j}D_{K}(x, y, t)$ as follows:
\begin{align}
&\p_{x}^{l}\p_{y}^{j}D_{K}(x, y, t) = -\frac{1}{p+1}\int_{0}^{t} \p_{x}^{l}\left\{ \p_{x}K(t-\tau)* \left(\partial_y^j\left(u^{p+1}(\tau)\right)\right)\right\} d\tau \nonumber \\
&= -\frac{1}{p+1}\int_{\frac{\delta t}{2}}^{t} K(t-\tau)* \left(\p_{x}^{l+1}\p_{y}^{j}\left(u^{p+1}(\tau)\right)\right) d\tau \nonumber\\
&\ \ \ \, -\frac{1}{p+1}\int_{0}^{\frac{\delta t}{2}}\int_{\R^{2}} \p_{x}^{l+1}\left( K(x-z, y-w, t-\tau)-K(x, y-w, t-\tau) \right) \p_{y}^{j}\left(u^{p+1}(z, w, \tau)\right) dzdwd\tau \nonumber\\
&\ \ \ \, -\frac{1}{p+1}\int_{0}^{\frac{\delta t}{2}}\int_{\R^{2}} \p_{x}^{l+1}K(x, y-w, t-\tau) \p_{y}^{j}\left(u^{p+1}(z, w, \tau)\right) dzdwd\tau \nonumber \\
&=: R_{1}^{\delta}(x, y, t)+R_{2}^{\delta}(x, y, t)+L^{\delta}(x, y, t). \label{D_K-split-R+L}
\end{align}
For the above $R_{1}^{\delta}(x, y, t)$ and $R_{2}^{\delta}(x, y, t)$, we are able to prove the following lemma: 
\begin{lem}\label{lem.R1-R2-est}
Under the same assumptions in Theorem~\ref{thm.u-asymp}, we have 
\begin{equation}\label{R1-R2-est}
\lim_{t \to \infty}t^{\frac{7}{4}+\frac{l}{2}}\left\|R_{i}^{\delta}(\cdot, \cdot, t)\right\|_{L^{\infty}}=0, \ \ i=1, 2.
\end{equation}
\end{lem}
\begin{proof}
First, we shall prove \eqref{R1-R2-est} for $i=1$. Applying Lemmas~\ref{lem.Duhamel-est} and \ref{lem.u^p+1-est}, we obtain
\begin{align*}
\left\|R_{1}^{\delta}(\cdot, \cdot, t)\right\|_{L^{\infty}}
&\le C\sqrt{t} \sup_{\frac{\delta t}{2} \le \tau \le t}\left( \left\|\p_{x}^{l+1}\p_{y}^{j}(u^{p+1}(\cdot, \cdot, \tau))\right\|_{L^{1}}+\left\|\p_{x}^{l+1}\p_{y}^{j+2}(u^{p+1}(\cdot, \cdot, \tau))\right\|_{L^{1}} \right) \\
&\le C\sqrt{t} \sup_{\frac{\delta t}{2} \le \tau \le t} \tau^{-\frac{l+1}{2}-\frac{3}{4}(p+1)} 
\le Ct^{-\frac{3}{4}(p+1)-\frac{l}{2}}, \ \ t\ge \frac{2}{\delta}. 
\end{align*}
Therefore, it follows from $p\ge2$ that 
\[
\limsup_{t \to \infty}t^{\frac{7}{4}+\frac{l}{2}}\left\|R_{1}^{\delta}(\cdot, \cdot, t)\right\|_{L^{\infty}} 
\le C \lim_{t\to \infty} t^{-\frac{3}{4}(p-\frac{4}{3})}=0. 
\]
This means that the desired result \eqref{R1-R2-est} holds for $i=1$. 

Next, we would like to show that \eqref{R1-R2-est} is true for $i=2$. Now, let us take a constant $\alpha>0$ and split the $z$-integral in $R_{2}^{\delta}(x, y, t)$ as follows: 
\begin{align}
R_{2}^{\delta}(x, y, t)&=-\frac{1}{p+1}\int_{0}^{\frac{\delta t}{2}}\int_{\R}\int_{|z|\le t^{\alpha}}\p_{x}^{l+1}\tilde{K}(x, z, y-w, t-\tau)\p_{y}^{j}\left(u^{p+1}(z, w, \tau)\right)dzdwd\tau \nonumber \\
&\ \ \ \, -\frac{1}{p+1}\int_{0}^{\frac{\delta t}{2}}\int_{\R}\int_{|z|\ge t^{\alpha}}\p_{x}^{l+1}\tilde{K}(x, z, y-w, t-\tau)\p_{y}^{j}\left(u^{p+1}(z, w, \tau)\right)dzdwd\tau \nonumber \\
&=: R_{2.1}^{\delta}(x, y, t)+R_{2.2}^{\delta}(x, y, t), \label{R2-split}
\end{align}
where $\tilde{K}(x, z, y-w, t-\tau)$ is defined by 
\[
\tilde{K}(x, z, y-w, t-\tau):=K(x-z, y-w, t-\tau)-K(x, y-w, t-\tau). 
\]

For $R_{2.1}^{\delta}(x, y, t)$, from the mean value theorem, there exists $\theta=\theta(x, z) \in (0, 1)$ such that 
\[
\p_{x}^{l+1}\tilde{K}(x, z, y-w, t-\tau)=-z(\p_{x}^{l+2}K)(x-\theta z, y-w, t-\tau)
\]
Therefore, it follows from Proposition~\ref{prop.L-decay-K}, Corollary~\ref{M-uniform} and $0<\delta<1$ that 
\begin{align}
\left|R_{2.1}^{\delta}(x, y, t)\right| 
&\le \frac{1}{p+1}\int_{0}^{\frac{\delta t}{2}}\int_{\R}\int_{|z|\le t^{\alpha}} |z| \left|\p_{x}^{l+2}K(x-\theta z, y-w, t-\tau)\right| \left|\p_{y}^{j}\left(u^{p+1}(z, w, \tau)\right)\right| dzdwd\tau \nonumber \\
&\le C\int_{0}^{\frac{\delta t}{2}}\int_{\R}\int_{|z|\le t^{\alpha}} t^{\alpha} (t-\tau)^{-\frac{5}{4}-\frac{l+2}{2}} \left|\p_{y}^{j}\left(u^{p+1}(z, w, \tau)\right)\right| dzdwd\tau \nonumber \\
&\le C\left( 1-\frac{\delta}{2} \right)^{-\frac{5}{4}-\frac{l+2}{2}} t^{\alpha-\frac{5}{4}-\frac{l+2}{2}}\int_{0}^{\infty}\int_{\R^{2}}\left|\p_{y}^{j}\left(u^{p+1}(z, w, \tau)\right)\right| dzdwd\tau \nonumber \\
&\le Ct^{\alpha-\frac{9}{4}-\frac{l}{2}}, \ \ (x, y) \in \R^{2}, \ t>0. \nonumber 
\end{align}
Thus, taking $\alpha>0$ to be $\alpha<\frac{1}{2}$, we obtain 
\begin{equation}\label{R2.1-est}
\limsup_{t \to \infty}t^{\frac{7}{4}+\frac{l}{2}}\left\|R_{2.1}^{\delta}(\cdot, \cdot, t)\right\|_{L^{\infty}} 
\le C \lim_{t\to \infty} t^{-(\frac{1}{2}-\alpha)}=0. 
\end{equation}

Analogously, for $R_{2.2}^{\delta}(x, y, t)$, we have 
\begin{align}
\left|R_{2.2}^{\delta}(x, y, t)\right| 
&\le \frac{1}{p+1}\int_{0}^{\frac{\delta t}{2}}\int_{\R}\int_{|z|\ge t^{\alpha}} \left|\p_{x}^{l+1}\tilde{K}(x, z, y-w, t-\tau)\right| \left|\p_{y}^{j}\left(u^{p+1}(z, w, \tau)\right)\right| dzdwd\tau \nonumber \\
&\le C\int_{0}^{\frac{\delta t}{2}}\int_{\R}\int_{|z|\ge t^{\alpha}} (t-\tau)^{-\frac{5}{4}-\frac{l+1}{2}} \left|\p_{y}^{j}\left(u^{p+1}(z, w, \tau)\right)\right| dzdwd\tau \nonumber \\
&\le C\left( 1-\frac{\delta}{2} \right)^{-\frac{5}{4}-\frac{l+1}{2}} t^{-\frac{5}{4}-\frac{l+1}{2}}\int_{0}^{\infty}\int_{\R^{2}} \mathbbm{1}_{\{|z|\ge t^{\alpha}\}}(z) \left|\p_{y}^{j}\left(u^{p+1}(z, w, \tau)\right)\right| dzdwd\tau \nonumber  \\
&\le Ct^{-\frac{7}{4}-\frac{l}{2}}\int_{0}^{\infty}\int_{\R^{2}} \mathbbm{1}_{\{|z|\ge t^{\alpha}\}}(z) \left|\p_{y}^{j}\left(u^{p+1}(z, w, \tau)\right)\right| dzdwd\tau, \ \ (x, y) \in \R^{2}, \ t>0. \nonumber 
\end{align}
Thus, from Corollary~\ref{M-uniform}, applying the Lebesgue dominated convergence theorem, we can see that 
\begin{align}
&\limsup_{t \to \infty}t^{\frac{7}{4}+\frac{l}{2}}\left\|R_{2.2}^{\delta}(\cdot, \cdot, t)\right\|_{L^{\infty}} \nonumber \\
&\le C \lim_{t\to \infty} \int_{0}^{\infty}\int_{\R^{2}} \mathbbm{1}_{\{|z|\ge t^{\alpha}\}}(z) \left|\p_{y}^{j}\left(u^{p+1}(z, w, \tau)\right)\right| dzdwd\tau=0. \label{R2.2-est}
\end{align}
Summing up \eqref{R2-split}, \eqref{R2.1-est} and \eqref{R2.2-est}, the desired estimate \eqref{R1-R2-est} for $i=2$ has been proven. 
\end{proof}

From Lemma~\ref{lem.R1-R2-est}, in order to prove Proposition~\ref{prop.w-main-lower}, it is sufficient to show 
\begin{equation*}
\left\|\p_{x}^{l}\mathcal{K}_{j}(\cdot, \cdot, t)+L^{\delta}(\cdot, \cdot, t)\right\|_{L^{\infty}} \ge \tilde{C}\left|\mathcal{M}_{j}\right|t^{-\frac{7}{4}-\frac{l}{2}}
\end{equation*}
for some positive constant $\tilde{C}$ and sufficiently large $t>0$, where $\mathcal{K}_{j}(x, y, t)$ and $L^{\delta}(x, y, t)$ are defined by \eqref{DEF-mathK} and \eqref{D_K-split-R+L}, respectively. In particular, it is enough to evaluate $(\p_{x}^{l}\mathcal{K}_{j}+L^{\delta})(0, 0, t)$ from below. Now, for the simplicity, we set 
\begin{align}
&U_{j}(w, \tau):=-\frac{1}{p+1}\int_{\R}\p_{y}^{j}\left( u^{p+1}(z, w, \tau) \right)dz, \ \ w \in \R, \ \tau>0, \label{DEF-Uj} \\
&Q_{j}(w):=M_{j}(w)+\int_{0}^{\infty}U_{j}(w, \tau) d\tau, \ \ w \in \R, \label{DEF-Qj}
\end{align} 
where $M_{j}(w)$ is defined by \eqref{M(y)}. Under the above notations, we get the following facts: 
\begin{lem}\label{lem.Uj-L1}
Under the same assumptions in Theorem~\ref{thm.u-asymp}, we have 
\begin{equation}\label{Uj-L1}
U_{j} \in L^{1}\left( \R \times [0, \infty) \right), \quad M_{j} \in L^{1}(\R), \quad Q_{j} \in L^{1}(\R). 
\end{equation}
In particular, the following relation holds: 
\begin{equation}\label{Uj-M}
\mathcal{M}_{j}=\int_{\R}Q_{j}(w)dw=\int_{\R}M_{j}(w)dw+\int_{0}^{\infty}\int_{\R}U_{j}(w, \tau)dwd\tau, 
\end{equation}
where $\mathcal{M}_{j}$ is defined by \eqref{DEF-mathM}. 
\end{lem}
\begin{proof}
From the assumptions on the initial data $u_{0}$ and Corollary~\ref{M-uniform}, we immediately have \eqref{Uj-L1}. 
The second Eq.~\eqref{Uj-M} directly follows from the definition of $\mathcal{M}_{j}$. 
\end{proof}
In what follows, let us deal with $(\p_{x}^{l}\mathcal{K}_{j}+L^{\delta})(0, 0, t)$. By using the above notations, we have 
\begin{align}
&(\p_{x}^{l}\mathcal{K}_{j}+L^{\delta})(0, 0, t) \nonumber \\
&=\int_{\R}\left( \p_{x}^{l+1}K(0, -w, t)\right)M_{j}(w)dw + \int_{0}^{\frac{\delta t}{2}}\int_{\R}\left( \p_{x}^{l+1}K(0, -w, t-\tau)\right)U_{j}(w, \tau)dw d\tau \nonumber \\
&=\left\{ \int_{\R}\left( \p_{x}^{l+1}K(0, -w, t)\right)M_{j}(w)dw + \int_{0}^{\frac{\delta t}{2}}\int_{\R}\left( \p_{x}^{l+1}K(0, -w, t)\right)U_{j}(w, \tau)dw d\tau \right\}\nonumber \\
&\ \ \ \ +\int_{0}^{\frac{\delta t}{2}}\int_{\R}\left( \p_{x}^{l+1}K(0, -w, t-\tau)-\p_{x}^{l+1}K(0, -w, t) \right)U_{j}(w, \tau)dw d\tau \nonumber \\
&=\int_{\R}\p_{x}^{l+1}K(0, -w, t)Q_{j}(w)dw - \int_{\frac{\delta t}{2}}^{\infty}\int_{\R}\left( \p_{x}^{l+1}K(0, -w, t) \right) U_{j}(w, \tau)dwd\tau \nonumber \\
&\ \ \ \ +\int_{0}^{\frac{\delta t}{2}}\int_{\R}\left( \p_{x}^{l+1}K(0, -w, t-\tau)-\p_{x}^{l+1}K(0, -w, t) \right)U_{j}(w, \tau)dw d\tau \nonumber \\
&=:I_{1}(t)-I_{2}^{\delta}(t)+I_{3}^{\delta}(t). \label{mathK+L-split}
\end{align}
Here, from the definition of $K(x, y, t)$ (i.e. \eqref{DEF-K}), we remark that 
\begin{align}
\tilde{K}(w, t)
:=&\, \p_{x}^{l+1}K(0, -w, t) \nonumber \\
=&\, \p_{x}^{l+1}\left\{ \frac{  t^{ -\frac{5}{4} } }{ 4\pi^{ \frac{3}{2} } \nu^{\frac{3}{4}}} \int_{0}^{\infty} r^{-\frac{1}{4} }e^{ -r } \cos \left( x \sqrt{\frac{r}{\nu t}} + \frac{y^{2}}{4\e}\sqrt{\frac{r}{\nu}}t^{-\frac{3}{2}} - \frac{\pi}{4}\e \right) dr \right\}\Biggl|_{x=0,\, y=-w} \nonumber \\
=&\, \frac{  t^{ -\frac{7}{4}-\frac{l}{2} } }{ 4\pi^{ \frac{3}{2} } \nu^{\frac{l}{2}+\frac{5}{4}}} \int_{0}^{\infty} r^{\frac{l}{2}+\frac{1}{4} }e^{ -r } \cos \left( \frac{w^{2}}{4\e}\sqrt{\frac{r}{\nu}}t^{-\frac{3}{2}} - \frac{\pi}{4}\e +\frac{(l+1)\pi}{2} \right) dr. \label{K-tilde}
\end{align}
Moreover, in what follows, we set 
\begin{equation}\label{DEF-Theta}
c_{0}:=\frac{1}{ 4\pi^{ \frac{3}{2} } \nu^{\frac{l}{2}+\frac{5}{4}}}, \quad 
\Theta(w, t, r):=\frac{w^{2}}{4\e}\sqrt{\frac{r}{\nu}}t^{-\frac{3}{2}} - \frac{\pi}{4}\e +\frac{(l+1)\pi}{2}. 
\end{equation}
Now, let us evaluate $I_{1}(t)$, $I_{2}^{\delta}(t)$ and $I_{3}^{\delta}(t)$. For these terms, we can establish the following lemmas: 
\begin{lem}\label{lem.I2delta-est}
Under the same assumptions in Theorem~\ref{thm.u-asymp}, we have 
\begin{equation}\label{I2delta-est}
\lim_{t \to \infty}t^{\frac{7}{4}+\frac{l}{2}}\left|I_{2}^{\delta}(t)\right|=0. 
\end{equation}
\end{lem}
\begin{proof}
From \eqref{K-tilde}, we can easily obtain  
\[
\left|\tilde{K}(w, t)\right| \le c_{0}\, \Gamma \left( \frac{l}{2}+\frac{5}{4}\right)t^{-\frac{7}{4}-\frac{l}{2}}, \ \ w\in \R, \ t>0.
\]
Therefore, it follows from the definitions of $I_{2}^{\delta}(t)$ (i.e. \eqref{mathK+L-split}) and $\tilde{K}(w, t)$ (i.e. \eqref{K-tilde}) that 
\[
\limsup_{t \to \infty}t^{\frac{7}{4}+\frac{l}{2}}\left|I_{2}^{\delta}(t)\right|
\le c_{0}\, \Gamma \left( \frac{l}{2}+\frac{5}{4}\right) \lim_{t\to \infty}\int_{\frac{\delta t}{2}}^{\infty} \int_{\R}\left|U_{j}(w, \tau)\right|dwd\tau =0. 
\]
Here, we have used the fact $U_{j} \in L^{1}\left( \R \times [0, \infty) \right)$, comes from \eqref{Uj-L1}. This completes the proof. 
\end{proof}

\begin{lem}\label{lem.I3delta-est}
Under the same assumptions in Theorem~\ref{thm.u-asymp}, we have 
\begin{equation}\label{I3delta-est}
\limsup_{t \to \infty}t^{\frac{7}{4}+\frac{l}{2}}\left|I_{3}^{\delta}(t)\right| \le C\delta. 
\end{equation}
\end{lem}
\begin{proof}
First we split the $w$-integral in $I_{3}^{\delta}(t)$ as follows:
\begin{align}
I_{3}^{\delta}(t)&=\int_{0}^{\frac{\delta t}{2}}\int_{|w|\ge t^{\frac{3}{4}}}\left( \tilde{K}(w, t-\tau)-\tilde{K}(w, t) \right) U_{j}(w, \tau)dwd\tau \nonumber \\
&\ \ \ +\int_{0}^{\frac{\delta t}{2}}\int_{|w|\le t^{\frac{3}{4}}}\left( \tilde{K}(w, t-\tau)-\tilde{K}(w, t) \right) U_{j}(w, \tau)dwd\tau 
=:I_{3.1}^{\delta}(t)+I_{3.2}^{\delta}(t). \label{I3delta-split}
\end{align}

For $I_{3.1}^{\delta}(t)$, according to \eqref{K-tilde}, we obtain 
\[
\left| \tilde{K}(w, t-\tau)-\tilde{K}(w, t) \right| \le c_{0}\, \Gamma \left( \frac{l}{2}+\frac{5}{4}\right)\left\{ (t-\tau)^{-\frac{7}{4}-\frac{l}{2}} + t^{-\frac{7}{4}-\frac{l}{2}} \right\}, \ \ w\in \R, \ t>0.
\]
Therefore, we have
\[
\left|I_{3.1}^{\delta}(t)\right| \le c_{0}\, \Gamma \left( \frac{l}{2}+\frac{5}{4}\right)\left\{ \left(1-\frac{\delta}{2}\right)^{-\frac{7}{4}-\frac{l}{2}} + 1 \right\}t^{-\frac{7}{4}-\frac{l}{2}} \int_{0}^{\frac{\delta t}{2}}\int_{|w|\ge t^{\frac{3}{4}}} \left|U_{j}(w, \tau)\right|dwd\tau, \ \ t>0.
\]
Here, we note that $U_{j} \in L^{1}\left( \R \times [0, \infty) \right)$ from \eqref{Uj-L1}. 
Thus, applying the Lebesgue dominated convergence theorem, we get 
\begin{equation}\label{I3.1delta-est}
\limsup_{t \to \infty}t^{\frac{7}{4}+\frac{l}{2}}\left|I_{3.1}^{\delta}(t)\right|
\le C\lim_{t\to \infty}\int_{0}^{\infty} \int_{|w|\ge t^{\frac{3}{4}}}\left|U_{j}(w, \tau)\right|dwd\tau =0. 
\end{equation}

On the other hand, from the mean value theorem, there exists $\theta=\theta(t, \tau)\in  (0,1)$ such that 
\begin{equation}\label{tildeK-mean}
\tilde{K}(w, t-\tau)-\tilde{K}(w, t)=-\tau(\p_{t}\tilde{K})(w, t-\theta \tau). 
\end{equation}
Moreover, it follows from the definition of $\tilde{K}(w, t)$ (i.e. \eqref{K-tilde}) that 
\begin{align*}
\p_{t}\tilde{K}(w, t)
&=c_{0}\left(-\frac{7}{4}-\frac{l}{2}\right)t^{-\frac{11}{4}-\frac{l}{2}}\int_{0}^{\infty}r^{\frac{l}{2}+\frac{1}{4}}e^{-r} \cos \left(\Theta(w, t, r)\right)dr \\
&\ \ \ \ +\frac{3c_{0}w^{2}}{8\e \sqrt{\nu}}t^{-\frac{17}{4}-\frac{l}{2}}\int_{0}^{\infty}r^{\frac{l}{2}+\frac{3}{4}}e^{-r} \sin \left(\Theta(w, t, r)\right)dr, 
\end{align*}
where $c_{0}$ and $\Theta(w, t, r)$ are defined by \eqref{DEF-Theta}. Therefore, if $|w|\le t^{\frac{3}{4}}$, we have 
\begin{equation}\label{tildeK-mean-est}
\left|\p_{t}\tilde{K}(w, t)\right|\le c_{0}\left(\frac{7}{4}+\frac{l}{2}\right)t^{-\frac{11}{4}-\frac{l}{2}}\Gamma\left(\frac{l}{2}+\frac{5}{4}\right)
+\frac{3c_{0}t^{\frac{3}{2}}}{8\e \sqrt{\nu}}t^{-\frac{17}{4}-\frac{l}{2}}\Gamma\left(\frac{l}{2}+\frac{7}{4}\right) 
\le Ct^{-\frac{11}{4}-\frac{l}{2}}. 
\end{equation}
Hence, from the definition of $I_{3.2}^{\delta}(t)$ (i.e. \eqref{I3delta-split}), \eqref{tildeK-mean} and \eqref{tildeK-mean-est}, we arrive at 
\begin{align*}
\left|I_{3.2}^{\delta}(t)\right| 
&\le C\int_{0}^{\frac{\delta t}{2}}\int_{|w|\le t^{\frac{3}{4}}} |\tau|(t-\theta \tau)^{-\frac{11}{4}-\frac{l}{2}}\left|U_{j}(w, \tau)\right|dwd\tau \\
&\le C\cdot \frac{\delta}{2}\left(1-\frac{\delta}{2}\right)^{-\frac{11}{4}-\frac{l}{2}} t^{-\frac{7}{4}-\frac{l}{2}}\int_{0}^{\frac{\delta t}{2}}\int_{|w|\le t^{\frac{3}{4}}}\left|U_{j}(w, \tau)\right|dwd\tau, \ \ t>0.
\end{align*}
Finally, using the fact $U_{j} \in L^{1}\left( \R \times [0, \infty) \right)$ which comes from \eqref{Uj-L1} again, we obtain 
\begin{equation}\label{I3.2delta-est}
\limsup_{t \to \infty}t^{\frac{7}{4}+\frac{l}{2}}\left|I_{3.2}^{\delta}(t)\right| \le C\delta. 
\end{equation}
Combining \eqref{I3delta-split}, \eqref{I3.1delta-est} and \eqref{I3.2delta-est}, we eventually arrive at \eqref{I3delta-est}. This completes the proof. 
\end{proof}

\begin{lem}\label{lem.I1-est}
Under the same assumptions in Theorem~\ref{thm.u-asymp}, there exist positive constants $C_{0}>0$ and $T_{0}>0$ such that the following estimate holds: 
\begin{equation}\label{I1-est-lower}
\left|I_{1}(t)\right|  \ge C_{0}\left|\mathcal{M}_{j}\right|t^{-\frac{7}{4}-\frac{l}{2}}, \ \ t \ge T_{0}, 
\end{equation}
where $I_{1}(t)$ and $\mathcal{M}_{j}$ are defined by \eqref{mathK+L-split} and \eqref{DEF-mathM}, respectively. 
\end{lem}
\begin{proof}
First, we take $\kappa>0$ and then split the integral in $I_{1}(t)$ as follows: 
\begin{align}
I_{1}(t) 
&= \int_{|w|\ge \kappa t^{\frac{3}{4}}}\p_{x}^{l+1}K(0, -w, t)Q_{j}(w)dw + \int_{|w|\le \kappa t^{\frac{3}{4}}}\p_{x}^{l+1}K(0, -w, t)Q_{j}(w)dw \nonumber \\
&=: I_{1.1}(t) + I_{1.2}(t). \label{I1-split}
\end{align}

Now, let us evaluate $I_{1.1}(t)$. From \eqref{K-tilde} and the definition of $Q_{j}(w)$ (i.e. \eqref{DEF-Qj}), we have 
\[
\left|I_{1.1}(t)\right| \le c_{0}\, \Gamma \left(\frac{l}{2}+\frac{5}{4}\right) t^{-\frac{7}{4}-\frac{l}{2}} \left( \int_{|w|\ge \kappa t^{\frac{3}{4}}} \left|M_{j}(w)\right| dw + \int_{0}^{\infty}\int_{|w|\ge \kappa t^{\frac{3}{4}}} \left| U_{j}(w, \tau) \right| dwd\tau \right), \ \ t>0,
\]
where $M_{j}(w)$ and $U_{j}(w, \tau)$ are defined by \eqref{M(y)} and \eqref{DEF-Uj}, respectively. 
Here, from Lemma~\ref{lem.Uj-L1}, we note that $M_{j} \in L^{1}(\R)$ and $U_{j} \in L^{1}(\R \times [0, \infty))$. 
Therefore, we obtain 
\begin{equation*}
\limsup_{t \to \infty}t^{\frac{7}{4}+\frac{l}{2}}\left|I_{1.1}(t)\right|=0. 
\end{equation*}
Thus, for any $\eta>0$, there exists $T_{1}>0$ such that the following estimate holds: 
\begin{equation}\label{I1.1-est}
t^{\frac{7}{4}+\frac{l}{2}}\left|I_{1.1}(t)\right| \le \eta, \ \ t\ge T_{1}.  
\end{equation}

Next, we would like to deal with $I_{1.2}(t)$. In what follows, for simplicity, we set 
\begin{equation*}
\alpha:=-\frac{\pi}{4}\e + \frac{(l+1)\pi}{2}. 
\end{equation*}
Before doing evaluate $I_{1.2}(t)$, we transform $\cos \left( \Theta (w, t, r)\right)$, where $\Theta(w, t, r)$ is defined by \eqref{DEF-Theta}. 
By using properties of the trigonometric functions and the mean value theorem, we obtain
\begin{align}
\cos \left( \Theta (w, t, r)\right)
&=\cos\left( \frac{w^{2}}{4\e}\sqrt{\frac{r}{\nu}}t^{-\frac{3}{2}} \right)\cos \alpha - \sin\left( \frac{w^{2}}{4\e}\sqrt{\frac{r}{\nu}}t^{-\frac{3}{2}} \right)\sin \alpha \nonumber \\
&=\left\{ 1- \frac{w^{2}}{4\e}\sqrt{\frac{r}{\nu}}t^{-\frac{3}{2}} \sin\left( \frac{\theta w^{2}}{4\e}\sqrt{\frac{r}{\nu}}t^{-\frac{3}{2}} \right) \right\} \cos \alpha
- \sin\left( \frac{w^{2}}{4\e}\sqrt{\frac{r}{\nu}}t^{-\frac{3}{2}} \right)\sin \alpha. \label{cosTheta}
\end{align}
Here, $\theta=\theta(w, t, r)\in (0, 1)$ came from the mean value theorem. 
Therefore, noting 
\[
\left|\sin\left( \frac{\theta w^{2}}{4\e}\sqrt{\frac{r}{\nu}}t^{-\frac{3}{2}} \right)\right| \le 1, \quad 
\left|\sin\left( \frac{w^{2}}{4\e}\sqrt{\frac{r}{\nu}}t^{-\frac{3}{2}} \right)\right| \le \frac{w^{2}}{4|\e|}\sqrt{\frac{r}{\nu}}t^{-\frac{3}{2}}, \ \ w\in \R, \ t>0, \ r>0 
\]
and the definition of $I_{1.2}(t)$ (i.e. \eqref{I1-split}), it follows from \eqref{K-tilde}, \eqref{DEF-Theta} and \eqref{cosTheta} that 
\begin{align}
\left|I_{1.2}(t)\right|
&= \left| \int_{|w|\le \kappa t^{\frac{3}{4}}} \left( c_{0}t^{-\frac{7}{4}-\frac{l}{2}}\int_{0}^{\infty}r^{\frac{l}{2}+\frac{1}{4}}e^{-r}\cos \left( \Theta (w, t, r)\right)dr \right) Q_{j}(w) dw \right| \nonumber \\
&\ge c_{0}t^{-\frac{7}{4}-\frac{l}{2}}
\Biggl\{ |\cos \alpha|\,\Gamma\left( \frac{l}{2}+\frac{5}{4} \right) \left| \int_{|w|\le \kappa t^{\frac{3}{4}}} Q_{j}(w)dw \right| \nonumber \\
&\ \ \ -\left(|\cos \alpha|+|\sin \alpha|\right)\,\Gamma\left( \frac{l}{2}+\frac{7}{4} \right)\frac{t^{-\frac{3}{2}}}{4|\e|\sqrt{\nu}} \left| \int_{|w|\le \kappa t^{\frac{3}{4}}} w^{2}Q_{j}(w)dw \right| \Biggl\}, \ \ t>0. \label{I1.2-est-pre1}
\end{align}
Moreover, since 
\[
\left| \int_{|w|\le \kappa t^{\frac{3}{4}}} w^{2}Q_{j}(w)dw \right| \le \kappa^{2}t^{\frac{3}{2}}\left\|Q_{j}\right\|_{L^{1}}
\]
and $|\cos \alpha|+|\sin \alpha|\le 2$, 
we have from \eqref{I1.2-est-pre1} that 
\begin{equation}\label{I1.2-est-pre2}
t^{\frac{7}{4}+\frac{l}{2}}\left|I_{1.2}(t)\right|
\ge c_{0}|\cos \alpha|\,\Gamma\left( \frac{l}{2}+\frac{5}{4} \right)\left| \int_{|w|\le \kappa t^{\frac{3}{4}}} Q_{j}(w)dw \right|
-\frac{\kappa^{2}\Gamma\left( \frac{l}{2}+\frac{7}{4} \right)}{2|\e|\sqrt{\nu}}\left\|Q_{j}\right\|_{L^{1}}, \ \ t>0.
\end{equation}
Here, recalling the definition of $Q_{j}(w)$ (i.e. \eqref{DEF-Qj}) and noting that 
\[
\lim_{t \to \infty} \int_{|w|\le \kappa t^{\frac{3}{4}}} Q_{j}(w)dw = \mathcal{M}_{j}, 
\]
we can see that there exists $T_{2}>0$ such that the following inequality holds: 
\begin{equation}\label{Qj^-int}
\left| \int_{|w|\le \kappa t^{\frac{3}{4}}} Q_{j}(w)dw \right| \ge \frac{\left|\mathcal{M}_{j}\right|}{2}, \ \ t\ge T_{2}. 
\end{equation}
Therefore, choosing the positive constant $\kappa>0$ which satisfies 
\begin{equation}\label{kappa-choose}
\frac{\kappa^{2}\Gamma\left( \frac{l}{2}+\frac{7}{4} \right)}{2|\e|\sqrt{\nu}}\left\|Q_{j}\right\|_{L^{1}}
=\frac{c_{0}|\cos \alpha| \Gamma\left( \frac{l}{2}+\frac{5}{4} \right)}{4}\left|\mathcal{M}_{j}\right|, 
\end{equation}
then it follows from \eqref{I1.2-est-pre2}, \eqref{Qj^-int} and \eqref{kappa-choose} that 
\begin{equation}\label{I1.2-est}
t^{\frac{7}{4}+\frac{l}{2}}\left|I_{1.2}(t)\right|
\ge \frac{c_{0}|\cos \alpha| \Gamma\left( \frac{l}{2}+\frac{5}{4} \right)}{4}\left|\mathcal{M}_{j}\right|, \ \ t\ge T_{2}. 
\end{equation}

Finally, if we choose the positive constant $\eta>0$ which satisfies 
\begin{equation}\label{eta-choose}
\eta = \frac{c_{0}|\cos \alpha| \Gamma\left( \frac{l}{2}+\frac{5}{4} \right)}{8}\left|\mathcal{M}_{j}\right| 
\end{equation}
and set 
\begin{equation}\label{DEF-C0-T0}
C_{0}:=\frac{c_{0}|\cos \alpha| \Gamma\left( \frac{l}{2}+\frac{5}{4} \right)}{8}, \quad T_{0}:=\max\{T_{1}, T_{2}\}, 
\end{equation}
then it follows from \eqref{I1-split}, \eqref{I1.1-est}, \eqref{I1.2-est}, \eqref{eta-choose} and \eqref{DEF-C0-T0} that 
\begin{align*}
t^{\frac{7}{4}+\frac{l}{2}}\left|I_{1}(t)\right|
&\ge t^{\frac{7}{4}+\frac{l}{2}}\left|I_{1.2}(t)\right|-t^{\frac{7}{4}+\frac{l}{2}}\left|I_{1.1}(t)\right| \\
&\ge \frac{c_{0}|\cos \alpha| \Gamma\left( \frac{l}{2}+\frac{5}{4} \right)}{4}\left|\mathcal{M}_{j}\right|-\frac{c_{0}|\cos \alpha| \Gamma\left( \frac{l}{2}+\frac{5}{4} \right)}{8}\left|\mathcal{M}_{j}\right|
=C_{0}\left|\mathcal{M}_{j}\right|, \ \ t\ge T_{0}. 
\end{align*}
Therefore, we have the desired estimate \eqref{I1-est-lower}. This completes the proof.  
\end{proof}

\begin{proof}[\rm{\bf{Proof of Proposition~\ref{prop.w-main-lower}}}]
First, based on the discussion given in the above, let us define
\[
Y^{\delta}(t):=\left|I_{2}^{\delta}(t)\right| + \left|I_{3}^{\delta}(t)\right| + \left\|R_{1}^{\delta}(\cdot, \cdot, t)\right\|_{L^{\infty}} + \left\|R_{2}^{\delta}(\cdot, \cdot, t)\right\|_{L^{\infty}}, \ \ t>0.
\]
Then, for any $\delta \in (0, 1)$, we have from \eqref{D_K-split-R+L} and \eqref{mathK+L-split} that 
\begin{align}
&\left\|\p_{x}^{l}\mathcal{K}_{j}(\cdot, \cdot, t)+\p_{x}^{l}\p_{y}^{j}D_{K}(\cdot, \cdot, t)\right\|_{L^{\infty}} \nonumber \\
&\ge \left\|\p_{x}^{l}\mathcal{K}_{j}(\cdot, \cdot, t)+L^{\delta}(\cdot, \cdot, t)\right\|_{L^{\infty}}
-\left( \left\|R_{1}^{\delta}(\cdot, \cdot, t)\right\|_{L^{\infty}} + \left\|R_{2}^{\delta}(\cdot, \cdot, t)\right\|_{L^{\infty}} \right) \nonumber \\
&\ge  \left|\p_{x}^{l}\mathcal{K}_{j}(0, 0, t)+L^{\delta}(0, 0, t)\right|
-\left( \left\|R_{1}^{\delta}(\cdot, \cdot, t)\right\|_{L^{\infty}} + \left\|R_{2}^{\delta}(\cdot, \cdot, t)\right\|_{L^{\infty}} \right) \nonumber \\
&\ge \left\{\left|I_{1}(t)\right|-\left(\left|I_{2}^{\delta}(t)\right| + \left|I_{3}^{\delta}(t)\right|\right) \right\}
-\left( \left\|R_{1}^{\delta}(\cdot, \cdot, t)\right\|_{L^{\infty}} + \left\|R_{2}^{\delta}(\cdot, \cdot, t)\right\|_{L^{\infty}} \right)  \nonumber \\
&= \left|I_{1}(t)\right|-Y^{\delta}(t), \ \ t>0. \label{w-main-lower-pre}
\end{align}
Moreover, from Lemmas~\ref{lem.R1-R2-est}, \ref{lem.I2delta-est} and \ref{lem.I3delta-est}, we can see that 
\[
\limsup_{t \to \infty} t^{\frac{7}{4}+\frac{l}{2}}Y^{\delta}(t) \le C\delta. 
\]
Thus, there exists $\tilde{T}_{0}>0$ such that the following estimate holds: 
\begin{equation}\label{Ydelta-est}
0 \le t^{\frac{7}{4}+\frac{l}{2}}Y^{\delta}(t) \le 2C\delta, \ \ t\ge \tilde{T}_{0}. 
\end{equation}

Finally, since $\delta>0$ can be chosen arbitrarily small, we shall take 
\begin{equation}\label{delta-choose}
\delta = \frac{C_{0}}{4C}\left|\mathcal{M}_{j}\right|. 
\end{equation}
Moreover, for $C_{0}>0$ and $T_{0}>0$ which are appeared in Lemma~\ref{lem.I1-est}, we set 
\begin{equation}\label{DEF-C*-T*}
C_{*}:=\frac{C_{0}}{2}, \quad T_{*}:=\max\{T_{0}, \tilde{T}_{0}\}. 
\end{equation}
Then, combining \eqref{w-main-lower-pre}, \eqref{I1-est-lower}, \eqref{Ydelta-est}, \eqref{delta-choose} and \eqref{DEF-C*-T*}, we eventually arrive at 
\[
t^{\frac{7}{4}+\frac{l}{2}}\left\|\p_{x}^{l}\mathcal{K}_{j}(\cdot, \cdot, t)+\p_{x}^{l}\p_{y}^{j}D_{K}(\cdot, \cdot, t)\right\|_{L^{\infty}}
\ge C_{0}\left|\mathcal{M}_{j}\right|-2C\delta=C_{*}\left|\mathcal{M}_{j}\right|, \ \ t\ge T_{*}. 
\]
Thus, we can see that the desired estimate \eqref{w-main-lower} is true. This completes the proof. 
\end{proof}

\begin{proof}[\rm{\bf{End of the Proof of Theorem~\ref{thm.u-est-lower}}}]
It follows from Theorems~\ref{thm.u-asymp}, \eqref{w-sol} and \eqref{Duhamel-main} that 
\begin{align}
\left\|\p_{x}^{l}\p_{y}^{j}u(\cdot, \cdot, t)\right\|_{L^{\infty}} 
&\ge \left\|\p_{x}^{l}\p_{y}^{j}w(\cdot, \cdot, t)\right\|_{L^{\infty}}-Ct^{-\frac{9}{4}-\frac{l}{2}} \nonumber \\
&\ge \left\|\p_{x}^{l}\mathcal{K}_{j}(\cdot, \cdot, t)+\p_{x}^{l}\p_{y}^{j}D_{K}(\cdot, \cdot, t)\right\|_{L^{\infty}} \nonumber \\
&\ \ \ \,-\left\|\p_{x}^{l}\p_{y}^{j}\left(K(t)*u_{0}\right)(\cdot, \cdot)-\p_{x}^{l}\mathcal{K}_{j}(\cdot, \cdot, t)\right\|_{L^{\infty}}-Ct^{-\frac{9}{4}-\frac{l}{2}}, \ \ t\ge2. \label{u-est-lower-final1}
\end{align}
By virtue of Theorem~\ref{thm.L-ap-LKPB}, there exists $\tilde{T}_{*}>0$ such that 
\begin{equation}\label{final-pre}
t^{\frac{7}{4}+\frac{l}{2}}\left\|\p_{x}^{l}\p_{y}^{j}\left(K(t)*u_{0}\right)(\cdot, \cdot)-\p_{x}^{l}\mathcal{K}_{j}(\cdot, \cdot, t)\right\|_{L^{\infty}} \le \frac{C_{*}}{3}\left|\mathcal{M}_{j}\right|, \ \ 
Ct^{-\frac{1}{2}} \le \frac{C_{*}}{3}\left|\mathcal{M}_{j}\right|, \ \ t\ge \tilde{T}_{*}, 
\end{equation}
where $C_{*}>0$ is the constant appeared in Proposition~\ref{prop.w-main-lower}. Moreover, for $T_{*}>0$ which is appeared in Proposition~\ref{prop.w-main-lower}, we set 
\begin{equation}\label{DEF-Cdag-Tdag}
C_{\dag}:=\frac{C_{*}}{3}, \quad T_{\dag}:=\max\{T_{*}, \tilde{T}_{*}, 2\}. 
\end{equation}
Therefore, summarizing up \eqref{u-est-lower-final1}, \eqref{final-pre}, Proposition~\ref{prop.w-main-lower} and \eqref{DEF-Cdag-Tdag}, we obtain 
\[
t^{\frac{7}{4}+\frac{l}{2}}\left\|\p_{x}^{l}\p_{y}^{j}u(\cdot, \cdot, t)\right\|_{L^{\infty}} \ge C_{*}\left|\mathcal{M}_{j}\right|-\frac{2C_{*}}{3}\left|\mathcal{M}_{j}\right|=C_{\dag}\left|\mathcal{M}_{j}\right|, \ \ t\ge T_{\dag}. 
\]
Thus, the desired estimate \eqref{u-est-lower} has been proven. This completes the proof of Theorem~\ref{thm.u-est-lower}. 
\end{proof}

\section*{Acknowledgments}
The first author is supported by Grant-in-Aid for Research Activity Start-up No.20K22303, Japan Society for the Promotion of Science. The second author is supported by Grant-in-Aid for Young Scientists Research~(B) No.17K14220, Japan Society for the Promotion of Science.



\bigskip
\par\noindent
\begin{flushleft}Ikki Fukuda\\
Division of Mathematics and Physics, \\
Faculty of Engineering, \\
Shinshu University\\
4-17-1, Wakasato, Nagano, 380-8553, JAPAN\\
E-mail: i\_fukuda@shinshu-u.ac.jp
\vskip15pt
Hiroyuki Hirayama\\
Faculty of Education, \\
University of Miyazaki, \\
1-1, Gakuenkibanadai-nishi, Miyazaki, 889-2192, JAPAN\\
E-mail: h.hirayama@cc.miyazaki-u.ac.jp
\end{flushleft}

\end{document}